\documentclass[11pt, reqno]{amsart}

\usepackage[utf8]{inputenc}
\usepackage[T1]{fontenc}
\usepackage{amsmath, amssymb, amsthm, mathrsfs}
\usepackage{mathtools}
\usepackage{geometry}
\usepackage{enumitem}
\usepackage[colorlinks=true, linkcolor=blue, citecolor=red, urlcolor=blue]{hyperref}
\usepackage{microtype} % Better typography
\usepackage{comment}

\newtheorem*{theorem*}{Theorem}
\newtheorem*{theoremA}{Theorem A}
\newtheorem*{theoremB}{Theorem B}
\newtheorem*{theoremC}{Theorem C}
\newtheorem*{corollaryD}{Corollary D}
\newtheorem{theorem}{Theorem}[section]
\newtheorem{lemma}[theorem]{Lemma}
\newtheorem{proposition}[theorem]{Proposition}
\newtheorem{corollary}[theorem]{Corollary}

\theoremstyle{definition}

\theoremstyle{remark}
\newtheorem{remark}[theorem]{Remark}

\newcommand{\D}{\mathbb{D}}
\newcommand{\T}{\mathbb{T}}
\newcommand{\C}{\mathbb{C}}

\newcommand{\R}{\mathbb{R}}

\newcommand{\CB}{\mathcal C_{\mathrm{bd}}}
\newcommand{\CK}{\mathcal C_{\mathrm{comp}}}
\newcommand{\CS}{\mathcal C_{\mathrm{ss}}}
\newcommand{\CSell}[1]{\mathcal C_{\ell^{#1}\text{-ss}}}
\newcommand{\CSL}[1]{\mathcal C_{L^{#1}\text{-ss}}}

\title[Fixed-Copy Exponents and Strict Singularity]
{Fixed-Copy Exponent Sets and Strict Singularity of Composition Operators Between Hardy Spaces}

\author[Y. Shi and S. Li]{Yecheng Shi and Songxiao Li}

\address{Yecheng Shi \\ School of Mathematics and Statistics, Lingnan Normal University, Zhanjiang 524048, Guangdong, China}
\email{09ycshi@sina.cn}

\address{Songxiao Li \\ Department of Mathematics, Shantou University, Guangdong 515063, China}
\email{jyulsx@163.com}



\begin{document}

\begin{abstract}
For a bounded operator \(T\) between Banach spaces, we introduce the
\emph{fixed-copy exponent sets}
\[
\begin{aligned}
\operatorname{Fix}_{\ell}(T)
&:=
\{r\ge1:T\text{ fixes a copy of }\ell^r\},\\
\operatorname{Fix}_{L}(T)
&:=
\{r\ge1:T\text{ fixes a copy of }L^r(0,1)\}.
\end{aligned}
\]
For \(1\le p,q<\infty\), we completely determine both sets for every
bounded composition operator \(C_\varphi:H^p\to H^q\). In particular,
\(C_\varphi\) is strictly singular if and only if
\(\operatorname{Fix}_{\ell}(C_\varphi)=\varnothing\).
The classification also gives complete characterizations of the
\(\ell^r\)-singular and \(L^r(0,1)\)-singular subclasses for every
\(r\ge1\). As a further consequence, it completely resolves
Problems~4.3\textup{(1)} and~4.3\textup{(2)} posed by Laitila,
Nieminen, Saksman, and Tylli.
The proofs introduce a new localization method for producing fixed
copies from boundary lower estimates. Its main ingredient is a
localization theorem independent of composition operators: for every
measurable \(E\subset\mathbb T\) with \(m(E)>0\) and
\(1\le p<r\le2\), it constructs a single copy of \(L^r(0,1)\) in
\(H^p\) on which lower \(L^s(E)\) estimates hold simultaneously for
all \(1\le s\le p\), with constants depending on \(E\) only through
\(m(E)\). The classification combines this method with pullback
measure criteria, known fixed-copy results, and classical subspace
restrictions. For \(r<2\), the localization theorem is proved using
stable integrals and analytic lifting; the endpoint \(r=2\) is
handled by an \(E\)-adapted lacunary construction.
\end{abstract}
\maketitle

\section{Introduction and main results}
\label{sec:intro}

Let \(\D\) denote the open unit disk, \(\T=\partial\D\) its boundary,
and \(m\) the normalized Lebesgue measure on \(\T\). For $1\le p<\infty$, $H^{p}$
stands for the classical Hardy space over $\D$. 
Unless otherwise specified, all Banach spaces are complex.
We write \(L^r(0,1)\) for \(L^r(0,1;\C)\), and write
\(L^r(0,1;\R)\) when the real space is intended.
Every analytic self-map
$\varphi:\D\to\D$ induces a composition operator $C_{\varphi}$ acting by
\[
(C_{\varphi}f)(z)=f(\varphi(z)),\qquad z\in\D.
\]
Littlewood's subordination principle ensures that
$C_{\varphi}:H^{p}\to H^{p}$ is bounded for all $1\le p<\infty$.
Comprehensive expositions on composition operators appear in Shapiro \cite{Shapiro1993} 
and Cowen--MacCluer~\cite{CowenMacCluer1995}; foundational aspects of operator theory on spaces of
analytic functions are developed in Zhu~\cite{Zhu2007}.

Our starting point is the relation between strict singularity and the
Banach spaces whose copies are fixed by \(C_\varphi\). Recall that a
bounded operator \(T:X\to Y\) is strictly singular if it is not bounded
below on any infinite-dimensional closed subspace of \(X\); see
Kato~\cite{Kato1958}. We say that \(T\) fixes a copy of a Banach space \(Z\) if there is a
closed subspace \(M\subset X\), isomorphic to \(Z\), on which \(T\)
is bounded below; equivalently, there is an isomorphic embedding
\(J:Z\to X\) such that \(TJ\) is bounded below. Following
\cite{LNST2017}, \(T\) is \(Z\)-singular if it fixes no copy of \(Z\). We ask the following questions: when is
\(C_\varphi:H^p\to H^q\) strictly singular, and for which exponents
\(r\) does \(C_\varphi\) fix a copy of \(\ell^r\) or \(L^r(0,1)\)?

For the diagonal case \(p=q\), Laitila, Nieminen, Saksman, and Tylli
proved that \(C_\varphi:H^p\to H^p\) is noncompact if and only if it
fixes a copy of \(\ell^p\)
\cite[Theorem~1.2]{LNST2017}; hence compactness and strict
singularity coincide
\cite[Corollary~1.3]{LNST2017}. Throughout, we set
\(
E_{\varphi}=\{\zeta\in\T:|\varphi^{*}(\zeta)|=1\},
\)
where \(\varphi^*\) denotes the radial boundary limit function of
\(\varphi\). For \(p\ne2\), the condition \(m(E_\varphi)>0\) is
equivalent to \(C_\varphi\) fixing a copy of \(\ell^2\)
\cite[Theorem~1.4]{LNST2017}, and, when \(1<p<\infty\), also to
fixing a copy of \(L^p(0,1)\)
\cite[Theorem~1.5]{LNST2017}. We say that \(\varphi\) has positive
boundary contact if \(m(E_\varphi)>0\). These results settle the
fixed-copy questions at the exponents \(p\) and \(2\), but do not
determine the full families \(\ell^r\) and \(L^r(0,1)\) as \(r\)
varies.

At \(p=1\), Sarason~\cite{Sarason1992} proved that weak compactness
and compactness coincide for composition operators on \(H^1\), while
Cima and Matheson~\cite{CimaMatheson1994} characterized complete
continuity by \(m(E_\varphi)=0\). Jarchow observed that a composition
operator on \(H^1\) is compact if and only if it is weakly
conditionally compact~\cite[p.~95]{Jarchow1998}; by Rosenthal's
\(\ell^1\)-theorem~\cite{Rosenthal1974}, this is equivalent to fixing
no copy of \(\ell^1\).

Two problems in~\cite{LNST2017} are particularly relevant here. Problem~4.3(1) asks whether the main theorems of~\cite{LNST2017} extend
to bounded composition operators $C_{\varphi}:H^{p}\to H^{q}$ with
$p\neq q$. Problem~4.3(2) asks whether there is an analogue of
Theorem~1.5 at \(p=1\). A straightforward
generalization of Theorem~1.5 to~$p=1$ is impossible, because $L^{1}(0,1)$
admits no isomorphic embedding into~$H^{1}$; see \cite[p.~262]{KwapienPelczynski1976} for a detailed
discussion.

The off-diagonal theory of composition operators between Hardy spaces
has been studied extensively. The integrability improving case was
studied by Hunziker and Jarchow~\cite{HunzikerJarchow1991}, while
boundedness and compactness for \(p\ne q\) were characterized in
\cite{ContrerasHernandezDiaz2003,BlascoJarchow2005}; related criteria in terms of the Nevanlinna counting function are
given in \cite{PerezGonzalezRattyaVukotic2007}. When \(q<p\),
Goebeler~\cite[Corollary~5, p.~390]{Goebeler2001} proved that
\(C_\varphi:H^p\to H^q\) is compact if and only if
\(m(E_\varphi)=0\). Essential norms and related quantitative questions
have also been studied; see
\cite{GorkinMacCluer2004,Demazeux2011,ShiLi2018,Bayart2024}.

For weighted composition operators, Lindstr\"om, Miihkinen, and
Nieminen~\cite{LindstromMiihkinenNieminen2020} proved, in particular,
that if \(1\le q\le p<\infty\), \(uC_\varphi:H^p\to H^q\) is bounded,
\(u\not\equiv0\), and \(m(E_\varphi)>0\), then \(uC_\varphi\) fixes a
copy of \(\ell^2\). Specializing to $u\equiv1$ recovers the case
$r=2$ for composition operators $C_{\varphi}:H^{p}\to H^{q}$ when $q<p$.
In the present paper we obtain this unweighted case directly from a fixed
dyadic lacunary subspace (Corollary~4.4). Combined with Goebeler's theorem,
this yields the equivalence of compactness and strict singularity for
$C_{\varphi}:H^{p}\to H^{q}$ when $q<p$.

This leaves the intermediate exponents \(p<r<2\), strict
singularity when \(p<q\), the endpoint \(r=p\) when \(q<p<2\),
and the \(L^r(0,1)\) fixed-copy question on \(H^1\).

\medskip\noindent\textbf{Fixed-copy exponent sets.}
For a bounded operator \(T:X\to Y\), we introduce the
\emph{fixed-copy exponent sets}
\[
\begin{aligned}
\operatorname{Fix}_\ell(T)
&:=
\{r\in[1,\infty):T\text{ fixes a copy of }\ell^r\},\\
\operatorname{Fix}_L(T)
&:=
\{r\in[1,\infty):T\text{ fixes a copy of }L^r(0,1)\}.
\end{aligned}
\]
Since \(L^r(0,1)\) contains \(\ell^r\) as a complemented subspace,
we have
\(\operatorname{Fix}_{L}(T)\subseteq\operatorname{Fix}_{\ell}(T)\)
for every bounded \(T\). If \(T\) is strictly singular, then both exponent sets are empty.
The converse fails in general: certain infinite-dimensional Banach
spaces contain no isomorphic copy of \(\ell^r\) or \(L^r(0,1)\) for
any \(1\le r<\infty\), and the identity operator on such a space
serves as a counterexample.
These exponent sets are invariant under Banach space isomorphisms and
are monotone under operator factorization.
If \(U:X_0\to X\) and \(V:Y\to Y_0\) are isomorphisms, then
\(\operatorname{Fix}_{\ell}(VTU)=\operatorname{Fix}_{\ell}(T)\) and
\(\operatorname{Fix}_{L}(VTU)=\operatorname{Fix}_{L}(T)\).
Moreover, if \(S=ATB\) with bounded linear maps \(A\) and \(B\), then
\(\operatorname{Fix}_{\ell}(S)\subseteq\operatorname{Fix}_{\ell}(T)\)
and
\(\operatorname{Fix}_{L}(S)\subseteq\operatorname{Fix}_{L}(T)\). They can also be expressed in
terms of \(Z\)-singularity:
\[
\begin{aligned}
r\notin\operatorname{Fix}_{\ell}(T)
&\Longleftrightarrow
T\text{ is }\ell^r\text{-singular},\\
r\notin\operatorname{Fix}_{L}(T)
&\Longleftrightarrow
T\text{ is }L^r(0,1)\text{-singular}.
\end{aligned}
\]

Johnson and Schechtman~\cite{JohnsonSchechtman2008} use the term
\(Z\)-strictly singular for \(Z\)-singular operators. In particular,
they proved that, for \(1<p<2\), every \(\ell^p\)-singular operator
\(T:L^p\to L^p\) is \(\ell^r\)-singular for all \(p<r<2\).
Equivalently,
\[
        r\in\operatorname{Fix}_{\ell}(T)
        \text{ for some }p<r<2
        \quad\Longrightarrow\quad
        p\in\operatorname{Fix}_{\ell}(T).
\]
Dosev, Johnson, and Schechtman~\cite{DosevJohnsonSchechtman2013}
showed that the \(L^p\)-singular operators on \(L^p\) form the largest
proper closed ideal in \(\mathcal L(L^p)\), linking
\(L^p\)-singularity to its ideal structure.

Another related point of comparison is provided by the
\(L\)- and \(V\)-characteristic sets of an operator
\(R:L^\infty(0,1)\to L^1(0,1)\).
The \(L\)-characteristic set records the pairs \((1/p,1/q)\) for which
\(R:L^p(0,1)\to L^q(0,1)\) is bounded
\cite{HernandezSemenovTradacete2017}, while the
\(V\)-characteristic set records those for which it is strictly
singular but not compact
\cite{HernandezSemenovTradacete2017,
HernandezSemenovTradacete2023}.
There the source and target exponents vary; here the ambient spaces
are fixed, and the varying parameter is the exponent \(r\) for which
the operator fixes a copy of \(\ell^r\) or \(L^r(0,1)\).

We now state the main results in terms of the fixed-copy exponent sets.
Theorem~A determines both fixed-copy exponent sets in the
off-diagonal case \(p\ne q\), while Theorem~C gives the diagonal
classification, including the endpoint \(p=1\). Together,
Theorems~A and~C completely determine both fixed-copy exponent sets
for every bounded composition operator
\(C_\varphi:H^p\to H^q\), \(1\le p,q<\infty\). In particular,
\[
        C_\varphi:H^p\to H^q
        \text{ is strictly singular}
        \quad\Longleftrightarrow\quad
        \operatorname{Fix}_{\ell}(C_\varphi)=\varnothing.
\]
Moreover, Theorems~A and~C completely resolve
Problems~4.3\textup{(1)} and~4.3\textup{(2)}
of~\cite{LNST2017}, respectively.
Theorem~B provides the main new mechanism behind the
intermediate-exponent parts of these classifications by producing
localized analytic copies of \(L^r(0,1)\) on sets of positive measure.

\subsection*{Classification I}  We first state the classification when \(p\ne q\).

\begin{theoremA}
Let \(1\le p,q<\infty\),  \(p\ne q\), and let \(\varphi:\D\to\D\) be analytic.  If \(p<q\), assume that
\(C_\varphi:H^p\to H^q\) is bounded.  Then
\[
\begin{array}{c|c|c|c}
\textup{Case}
& \textup{Condition}
& \operatorname{Fix}_{\ell}(C_\varphi)
& \operatorname{Fix}_{L}(C_\varphi)\\
\hline
q<p
& m(E_\varphi)=0
& \varnothing
& \varnothing\\
q<p<2
& m(E_\varphi)>0
& (p,2]
& (p,2]\\
q<p,\ p\ge2
& m(E_\varphi)>0
& \{2\}
& \{2\}\\
p<q
& C_\varphi\textup{ bounded}
& \varnothing
& \varnothing
\end{array}
\]
When \(q<p\), boundedness is automatic and
\[
C_\varphi:H^p\to H^q\text{ is compact}
\quad\Longleftrightarrow\quad
m(E_\varphi)=0
\quad\Longleftrightarrow\quad
C_\varphi\text{ is strictly singular}.
\]
When \(p<q\), boundedness forces \(m(E_\varphi)=0\), and every
bounded \(C_\varphi:H^p\to H^q\) is strictly singular.
\end{theoremA}

\subsection*{The localization theorem}

The known endpoint results do not by themselves yield the intermediate
exponents. On the diagonal, the known fixed copies of \(\ell^p\) and
\(\ell^2\) may occur on different subspaces and therefore provide no
direct mechanism for obtaining the exponents \(p<r<2\). The missing
ingredient is a copy of \(L^r(0,1)\) in \(H^p\) on which the
restriction operator to a prescribed measurable set
\(E\subset\T\) of positive measure is bounded below. Moreover, the
same copy must satisfy lower \(L^s(E)\)-estimates for every
\(1\le s\le p\), with constants depending on \(E\) only through
\(m(E)\). The following theorem provides this localization and is
independent of composition operators.

\begin{theoremB}
Let \(1\le p<r\le2\), and let \(E\subset\T\) be measurable with \(\delta=m(E)>0\). Then there exists a closed subspace
\[
        M_{p,r,E}\subset H^p,
        \qquad
        M_{p,r,E}\simeq L^r(0,1),
\]
such that, for every \(1\le s\le p\), there is a constant \(c_{p,r,s}(\delta)>0\), depending only on \(p\), \(r\), \(s\), and \(\delta\), satisfying
\[
        \|f\|_{L^s(E)}
        \ge
        c_{p,r,s}(\delta)\|f\|_{H^p},
        \qquad
        f\in M_{p,r,E}.
\]
\end{theoremB}

The existence of copies of \(L^r(0,1)\) without localization is
classical. Stable laws have long been used to realize \(L^r(0,1)\)
inside \(L^t\) spaces for \(1\le t<r\le2\); see
\cite{BretagnolleDacunhaCastelleKrivine1966}. The stable integral
used below is taken from Kanter
\cite[Sections~4--5 and the theorem, pp.~405--406]{Kanter1973}.
Kwapie\'n and Pe{\l}czy\'nski combined such realizations with the
Boas transform to obtain analytic copies of \(L^r(0,1)\)
\cite[Proposition~2.1 and its proof]{KwapienPelczynski1976};
see also Hardin~\cite{Hardin1982} for related stable embedding
results.

The new point in Theorem~B is localization on an arbitrary prescribed
measurable set \(E\), simultaneously for all \(1\le s\le p\), with
constants depending on \(E\) only through \(m(E)\). This quantitative
localization converts boundary lower estimates into fixed copies for
all intermediate exponents. The proof gives explicit lower bounds for
\(c_{p,r,s}(\delta)\); see Remark~\ref{rem:local-bound}.

By contrast, closed-range, reverse-Carleson, and related results
concern lower estimates on the whole Hardy space; see, for example,
\cite{CimaThomsonWogen1974,Zorboska1994,
LefevreLiQueffelecRodriguezPiazza2012,
HartmannMassanedaNicolauOrtegaCerda2014,
ChalendarPartington2019,Dyakonov2025}.
Indeed, when \(0<m(E)<1\), an outer-function argument shows that the
boundary restriction map
\(R_E:H^p\to L^s(E)\), \(R_Ef=f^*|_E\), is not bounded below for any
\(1\le s\le p\). Thus the localization in Theorem~B is necessarily a
subspace phenomenon.

\subsection*{Classification II}  We now state the diagonal classification.

\begin{theoremC}\label{thm:C}
Let \(1\le p<\infty\), and let \(\varphi:\D\to\D\) be analytic. Then
\[
\begin{array}{c|c|c|c}
\textup{Case}
& \textup{Range of }p
& \operatorname{Fix}_{\ell}(C_\varphi)
& \operatorname{Fix}_{L}(C_\varphi)\\
\hline
C_\varphi\textup{ compact}
& 1\le p<\infty
& \varnothing
& \varnothing\\
C_\varphi\textup{ noncompact}
& p=2
& \{2\}
& \{2\}\\
C_\varphi\textup{ noncompact},\ m(E_\varphi)=0
& p\ne2
& \{p\}
& \varnothing\\
m(E_\varphi)>0
& p=1
& [1,2]
& (1,2]\\
m(E_\varphi)>0
& 1<p<2
& [p,2]
& [p,2]\\
m(E_\varphi)>0
& 2<p<\infty
& \{2,p\}
& \{2,p\}
\end{array}.
\]
Moreover,
\[
        C_\varphi:H^p\to H^p\text{ is compact}
        \quad\Longleftrightarrow\quad
        C_\varphi:H^p\to H^p\text{ is strictly singular}.
\]
\end{theoremC}

For \(p=1\),
\(1\in\operatorname{Fix}_{\ell}(C_\varphi)\) if and only if
\(C_\varphi\) is noncompact, whereas
\(1\notin\operatorname{Fix}_{L}(C_\varphi)\) for every composition
operator \(C_\varphi:H^1\to H^1\). Consequently,
\(\operatorname{Fix}_{L}(C_\varphi)=(1,2]\) if
\(m(E_\varphi)>0\), and is empty if \(m(E_\varphi)=0\). The \(p=1\) part of Theorem~C completely settles
\cite[Problem~4.3(2)]{LNST2017}. In the case \(p>2\), classical structural constraints imposed on the subspaces of \(H^p\)  exclude every exponent other than
\(2\) and \(p\).

For each \(1\le r<\infty\), Theorems~A and~C determine which
composition operators fix no copy of \(\ell^r\) or \(L^r(0,1)\).
The corresponding relative class characterizations are given below.

For Banach spaces \(X\) and \(Y\), let \(\mathcal L(X,Y)\) denote the
space of bounded linear operators from \(X\) to \(Y\). Let
\(\mathcal K(X,Y)\) and \(\mathcal S(X,Y)\) denote the classes of
compact and strictly singular operators in \(\mathcal L(X,Y)\),
respectively. For \(1\le r<\infty\), let
\(\mathcal S_{\ell^r}(X,Y)\) and \(\mathcal S_{L^r}(X,Y)\) denote the
classes of operators in \(\mathcal L(X,Y)\) that fix no copy of
\(\ell^r\) and \(L^r(0,1)\), respectively. Thus
\[
T\in\mathcal S_{\ell^r}(X,Y)
\Longleftrightarrow
r\notin\operatorname{Fix}_{\ell}(T),
\qquad
T\in\mathcal S_{L^r}(X,Y)
\Longleftrightarrow
r\notin\operatorname{Fix}_{L}(T).
\]
For operators on a single Hardy space, one has the general identity
\[
        \mathcal S(H^p,H^p)
        =
        \mathcal S_{\ell^2}(H^p,H^p)
        \cap
        \mathcal S_{\ell^p}(H^p,H^p),
        \qquad
        1<p<\infty,\quad p\ne2;
\]
see \cite[(1.4)]{LNST2017}. Thus an operator on \(H^p\) is strictly
singular if and only if it fixes neither a copy of \(\ell^2\) nor a
copy of \(\ell^p\).

Within the class of composition operators, for parameters satisfying \(1\le p,q<\infty\), define
  \[
\begin{aligned}
\CB(H^p,H^q)
&:=
\bigl\{
C_\varphi\in\mathcal L(H^p,H^q):
\varphi:\D\to\D\text{ is analytic}
\bigr\},\\
\CK(H^p,H^q)
&:=
\CB(H^p,H^q)\cap\mathcal K(H^p,H^q),\\
\CS(H^p,H^q)
&:=
\CB(H^p,H^q)\cap\mathcal S(H^p,H^q),\\
\CSell{r}(H^p,H^q)
&:=
\CB(H^p,H^q)\cap\mathcal S_{\ell^r}(H^p,H^q),\\
\CSL{r}(H^p,H^q)
&:=
\CB(H^p,H^q)\cap\mathcal S_{L^r}(H^p,H^q),\\
\mathcal C_0(H^p,H^q)
&:=
\bigl\{
C_\varphi\in\CB(H^p,H^q):
m(E_\varphi)=0
\bigr\}.
\end{aligned}
\]
For the case \(p=q\), we adopt the shorthand notation
 \(
\mathcal C_0(H^p):=\mathcal C_0(H^p,H^p).
\)

\begin{corollaryD}
Let \(1\le p,q<\infty\). 

\begin{enumerate}[label=\textup{(\roman*)}]
\item If \(q<p\), then
\[
\CK(H^p,H^q)
=
\CS(H^p,H^q)
=
\mathcal C_0(H^p,H^q)
\subsetneq
\CB(H^p,H^q).
\]
For every \(1\le r<\infty\),
\[
\CSell{r}(H^p,H^q)
=
\CSL{r}(H^p,H^q)
=
\begin{cases}
\CK(H^p,H^q),
& 1<p<2,\quad p<r\le2,\\[1mm]
\CK(H^p,H^q),
& p\ge2,\quad r=2,\\[1mm]
\CB(H^p,H^q),
& \textup{otherwise}.
\end{cases}
\]

\item If \(p<q\), then, for every \(1\le r<\infty\),
\[
\begin{aligned}
\CK(H^p,H^q)
&\subsetneq
\CS(H^p,H^q)
=
\CSell{r}(H^p,H^q)\\
&=
\CSL{r}(H^p,H^q)
=
\mathcal C_0(H^p,H^q)
=
\CB(H^p,H^q).
\end{aligned}
\]

\item If \(p=q\), then
\[
\CK(H^p,H^p)
=
\CS(H^p,H^p)
\subsetneq
\mathcal C_0(H^p)
\subsetneq
\CB(H^p,H^p).
\]
Moreover, for every \(1\le r<\infty\),
\[
\CSell{r}(H^p,H^p)
=
\begin{cases}
\CK(H^p,H^p),
& r=p,\\[1mm]
\mathcal C_0(H^p),
& 1\le p<2,\quad p<r\le2,\\[1mm]
\mathcal C_0(H^p),
& 2<p<\infty,\quad r=2,\\[1mm]
\CB(H^p,H^p),
& \textup{otherwise},
\end{cases}
\]
whereas
\[
\CSL{r}(H^p,H^p)
=
\begin{cases}
\mathcal C_0(H^1),
& p=1,\quad 1<r\le2,\\[1mm]
\mathcal C_0(H^p),
& 1<p<2,\quad p\le r\le2,\\[1mm]
\CK(H^2,H^2),
& p=2,\quad r=2,\\[1mm]
\mathcal C_0(H^p),
& 2<p<\infty,\quad r\in\{2,p\},\\[1mm]
\CB(H^p,H^p),
& \textup{otherwise}.
\end{cases}
\]
\end{enumerate}
\end{corollaryD}

The class identities in Corollary~D follow from Theorems~A and~C.
Read together with the classical boundedness and compactness criteria
in Proposition~\ref{prop:classical}, these identities give concrete
characterizations of \(\ell^r\)-singularity and
\(L^r(0,1)\)-singularity for bounded composition operators
\(C_\varphi:H^p\to H^q\), for all \(1\le p,q,r<\infty\).
The strict inclusions in Corollary~D are proved in
Section~\ref{sec:lens}. In particular, when \(p<q\), critical lens
maps give bounded noncompact composition operators; hence
\(\CS(H^p,H^q)=\CB(H^p,H^q)\) means that every bounded composition
operator is strictly singular, not that every such operator is compact.

\subsection*{Comparison with the diagonal theory}

The fixed-copy exponent sets formulate a broader rigidity problem:
for each bounded composition operator, determine all exponents
\(r\) for which copies of \(\ell^r\) or \(L^r(0,1)\) are fixed.
While \cite{LNST2017} determines when composition operators
\(C_\varphi:H^p\to H^p\) fix copies of \(\ell^p\), \(\ell^2\),
and, for \(1<p<\infty\), \(p\ne2\), \(L^p(0,1)\),
Theorems~A and~C determine both fixed-copy exponent sets for every
bounded \(C_\varphi:H^p\to H^q\). Thus the extension is twofold:
from the specific exponents \(p\) and \(2\) to the full range of
fixed-copy exponents, and from the diagonal case to composition
operators between different Hardy spaces. The proofs also require
new mechanisms. A central new ingredient is Theorem~B, a localization
theorem independent of composition operators. Its proof combines stable integrals, localization on
prescribed measurable sets of positive measure, conjugation symmetry,
and analytic lifting through the Riesz projection to produce
localized analytic copies of \(L^r(0,1)\). For \(p<2\), it converts boundary lower estimates into fixed copies
of \(L^r(0,1)\) for every \(p<r\le2\), both in the diagonal case
and when \(q<p\). In the diagonal classification, these new fixed
copies are combined with the endpoint results of \cite{LNST2017},
while classical Banach space structure results exclude the remaining
exponents. For \(p<q\),
compactness after raising the source exponent is combined with
Kwapie\'n--Pe{\l}czy\'nski multiplier factorizations when \(p<2\),
with a Hilbert space argument when \(p=2\), and with
Kadec--Pe{\l}czy\'nski subspace arguments when \(p>2\), to prove
that every bounded \(C_\varphi:H^p\to H^q\) is strictly singular.

\subsection*{Outline of the proof of Theorem~B}

We treat \(r=2\) and \(p<r<2\) separately. For \(r=2\), an
\(E\)-adapted dyadic lacunary tail gives a copy of
\(\ell^2\simeq L^2(0,1)\) with the required lower estimates on \(E\).
For \(p<r<2\), an averaging argument gives a rotation
\(\widetilde E\) of \(E\) and a conjugation-invariant set
\(G\subset\widetilde E\) with \(m(G)\ge m(E)^2\). A
measure-preserving parametrization transfers the real \(r\)-stable
model arising from Kanter's integral to one half of \(G\), and
reflection extends it to the other half. The \(L^1\)-, \(L^p\)-,
and \(L^t\)-moment estimates, for some \(p<t<r\), provide
respectively the localized lower bound, the lower \(H^p\) control,
and the upper control needed for analytic lifting through the Riesz
projection. Conjugation symmetry then yields a complex linear copy
without losing the localized lower estimate, and a final rotation
returns the construction to \(E\).

\subsection*{Outline of the proofs of Theorems~A and~C}

\noindent\textit{The case \(q<p\).}
If \(m(E_\varphi)=0\), then \(C_\varphi:H^p\to H^q\) is compact.
If \(m(E_\varphi)>0\), the boundary part of the pullback measure
yields a measurable set \(E\subset\T\), \(m(E)>0\), such that
\(\|C_\varphi f\|_{H^q}
\gtrsim
\|f\|_{L^q(E)}.
\)
For \(p<2\), Theorem~B gives fixed copies of \(L^r(0,1)\), and hence
of \(\ell^r\), for every \(p<r\le2\); for \(p\ge2\), the fixed
dyadic lacunary space \(M_2\) gives \(r=2\).
Proposition~\ref{prop:ambient-ellr} excludes all other sequence
exponents except possibly \(r=p\) when \(p<2\), and Rosenthal's
strong-embedding theorem excludes this remaining endpoint.

\smallskip
\noindent\textit{The case \(p<q\).}
Boundedness implies \(m(E_\varphi)=0\) and compactness of
\(C_\varphi:H^p\to H^p\) and \(C_\varphi:H^u\to H^q\) for every
\(u>p\). If \(p<2\), assuming that \(C_\varphi\) is bounded below on
an infinite-dimensional subspace \(M\subset H^p\), the hypotheses of
the Kwapie\'n--Pe\l czy\'nski factorization are satisfied
(\(M\) is reflexive when \(p=1\), while for \(1<p<2\) it contains no
copy of \(\ell^p\)). Thus the inclusion \(M\hookrightarrow H^p\)
factors through some \(H^u\), \(u>p\), contradicting the compactness
of \(C_\varphi:H^u\to H^q\).
If \(p\ge2\), the Kadec--Pe\l czy\'nski subspace theorem reduces the
problem to a fixed copy of \(\ell^2\); the norm comparison in
Lemma~\ref{lem:kp-hilbert}, together with the compactness of
\(C_\varphi:H^p\to H^p\), again gives a contradiction.

\smallskip
\noindent\textit{The diagonal case \(p=q\).}
When \(m(E_\varphi)>0\), the known endpoint results, supplemented
by Theorem~B for \(p<2\), give the asserted fixed-copy ranges; the
only exceptional endpoint is \(p=1\), where \(r=1\) occurs only on
the sequence side because \(L^1(0,1)\) does not embed isomorphically
into \(H^1\). When \(m(E_\varphi)=0\) and \(C_\varphi\) is noncompact,
the multiplier factorizations exclude the additional exponents for
\(p<2\), while Proposition~\ref{prop:ambient-ellr} and the endpoint
results settle \(p>2\). At \(p=2\), noncompactness gives
\(\operatorname{Fix}_{\ell}(C_\varphi)
=\operatorname{Fix}_{L}(C_\varphi)=\{2\}\).

\subsection*{Organization of the paper}

%The paper is organized as follows. 
Section~\ref{sec:prelim}
collects the Banach space and Hardy space tools used later.
Section~\ref{sec:pullback} develops the framework of pullback measures,
records the classical boundedness and compactness criteria, and derives
the compactness consequences needed in the case \(p<q\). Section~\ref{sec:localization} constructs
localized analytic copies on sets of positive measure and proves
Theorem~B. Sections~\ref{sec:q-lt-p} and~\ref{sec:p-lt-q}
treat \(q<p\) and \(p<q\), respectively, and prove the corresponding
parts of Theorem~A. Section~\ref{sec:diag} proves Theorem~C.
Finally, Section~\ref{sec:lens} completes the proof of
Corollary~D and proves its strict inclusions.

Throughout, \(A\lesssim B\) means that \(A\le CB\), where the implicit constant \(C\) may depend on fixed parameters but is independent of the varying quantities under consideration.  We write
\(A\asymp B\) if both \(A\lesssim B\) and \(B\lesssim A\).

%%%%%%%%%%%%%%%%%%%%%%%%%%%%%%%%%%%%%%%%%%%%%%%%%%

\section{Banach and Hardy spaces}
\label{sec:prelim}

We collect three tools used below: restrictions on
\(\ell^r\)-subspaces of Hardy spaces, the multiplier factorizations
of Kwapie\'n and Pe{\l}czy\'nski, and the classical lacunary series
estimates.

\subsection{\texorpdfstring{Subspaces of Hardy and \(L^p\)-spaces} 
{Subspaces of Hardy and Lp-spaces}}
%We first recall two standard facts concerning classical sequence spaces and subspaces of \(L^p\).

\begin{lemma}
\label{lem:ellr-rigidity}
Let \(1\le r,s<\infty\). If \(\ell^r\) contains a closed subspace
isomorphic to \(\ell^s\), then \(r=s\).
\end{lemma}

\begin{proof}
If \(r\ne s\), no infinite-dimensional subspace of \(\ell^r\) is
isomorphic to a subspace of \(\ell^s\); see
Albiac--Kalton~\cite[Corollary~2.1.6]{AlbiacKalton2016}.
\end{proof}

\begin{lemma}
\label{lem:kp-subspace}
Let \(2<p<\infty\), and let \(Y\) be an infinite-dimensional closed
subspace of \(L^p(\T)\). Then \(Y\) contains a closed subspace
isomorphic either to \(\ell^2\) or to \(\ell^p\).
\end{lemma}

\begin{proof}
This is the Kadec--Pe{\l}czy\'nski subspace theorem; see
\cite[Corollary~2, p.~168]{KadecPelczynski1962}.
\end{proof}

We shall also use the following restriction on copies of \(\ell^r\)
in Hardy spaces.

\begin{proposition}
\label{prop:ambient-ellr}
Let \(1\le r,s<\infty\). If \(H^s\) contains a closed subspace
isomorphic to \(\ell^r\), then
\[
        s\le r\le2 \quad (1\le s\le2),
        \qquad
        r\in\{2,s\} \quad (2<s<\infty).
\]
\end{proposition}

\begin{proof}
The boundary-value map identifies \(H^s\) isometrically with a closed
subspace of \(L^s(\T)\). Thus every copy of \(\ell^r\) in \(H^s\)
is also a copy of \(\ell^r\) in \(L^s(\T)\). The conclusion follows
from the characterization of those \(r\) for which \(\ell^r\)
embeds into \(L^s\); see
Albiac--Kalton~\cite[Theorem~6.4.19]{AlbiacKalton2016}.
\end{proof}

We shall also use the following complex form of the
Kadec--Pe{\l}czy\'nski theorem.

\begin{lemma}
\label{lem:kp-hilbert}
Let \(2<q<\infty\), and let \(N\) be a closed complex subspace of
\(L^q(\T)\). If \(N\simeq\ell^2\), then the \(L^q\)- and
\(L^2\)-norms are equivalent on \(N\). In particular, there is a
constant \(A_N>0\) such that
\[
        \|g\|_{L^q(\T)}
        \le
        A_N\|g\|_{L^2(\T)},
        \qquad
        g\in N.
\]
\end{lemma}

\begin{proof}
For real subspaces, this is the implication
\textup{3a}\(\Rightarrow\)\textup{3e} in
\cite[Theorem~3, pp.~166--167]{KadecPelczynski1962},
after identifying \((\T,m)\) with the unit interval modulo null sets.
The complex case follows by applying the real result to the
realification of \(N\), realized by
\(f\mapsto(\Re f,\Im f)\) on two disjoint copies of \(\T\).
With the sum measure on the two copies, the \(L^2\)-norm is preserved
and the \(L^q\)-norm is equivalent to the original \(L^q\)-norm on
\(N\).
\end{proof}

\subsection{Multiplier factorizations}

The two factorizations below are due to
Kwapie\'n and Pe{\l}czy\'nski.

\begin{lemma}
\label{lem:kp-below2}
Let \(1<s<2\), and let \(X\) be a closed subspace of \(H^s\)
containing no closed subspace isomorphic to \(\ell^s\). Then there
exist \(t\in(s,2)\) and an outer function \(g\in H^a\), where \(a=st/(t-s)\), such that
\[
        U:X\to H^t,\quad Uf=\frac{f}{g},
        \qquad
        V:H^t\to H^s,\quad Vh=gh,
\]
are bounded and satisfy \(VU=j_X\), where
\(j_X:X\hookrightarrow H^s\) denotes the inclusion.
\end{lemma}

\begin{proof}
This is precisely the multiplier factorization in
\cite[Proposition~2.2a, p.~267]{KwapienPelczynski1976},
applied with \(p_0=s\). That proposition gives \(t_0\in(s,2)\)
such that, for every \(t\in(s,t_0)\), there is an outer function
\(g\in H^{st/(t-s)}
\)
for which the inclusion \(j_X:X\hookrightarrow H^s\) factors as
\[
        j_X=M_gM_{1/g},
\]
where
\[
        M_{1/g}:X\to H^t,
        \qquad
        M_g:H^t\to H^s
\]
are bounded. Choose \(t\in(s,t_0)\), and set
\(U=M_{1/g}\) and \(V=M_g\). Since \(g\) is outer, it has no zeros
in \(\D\), so \(Uf=f/g\) is analytic. Thus \(U\) and \(V\) have the
stated forms and satisfy \(VU=j_X\).
\end{proof}

The \(H^1\) case is the multiplier factorization constructed in the
proof of Kwapie\'n--Pe{\l}czy\'nski's Proposition~2.2.

\begin{lemma}
\label{lem:kp-H1}
Let \(X\) be a reflexive closed subspace of \(H^1\). Then there exist
\(t>1\), an outer function \(g\in H^a\), where
\(a=t/(t-1)\), and bounded complex linear operators
\[
        U:X\to H^t,\quad Uf=\frac{f}{g},
        \qquad
        V:H^t\to H^1,\quad Vh=gh,
\]
such that
\[
        g^{-1}\in H^\infty,
        \qquad
        \|g^{-1}\|_{H^\infty}\le1,
        \qquad
        VU=j_X,
\]
where \(j_X:X\hookrightarrow H^1\) denotes the inclusion.
\end{lemma}

\begin{proof}
The proof of
\cite[Proposition~2.2, pp.~265--266]{KwapienPelczynski1976}
gives \(t_1>1\) such that, for every \(t\in(1,t_1)\), one may choose
an outer function \(g\in H^{t/(t-1)}\) satisfying
\[
        |g(z)|\ge1,
        \qquad z\in\D,
\]
for which the multiplication operators
\(M_{1/g}:X\to H^t\) and \(M_g:H^t\to H^1\) are bounded and satisfy
\(M_gM_{1/g}=j_X\).

Choose \(t\in(1,t_1)\), and set
\(a=t/(t-1)\), \(U=M_{1/g}\), and \(V=M_g\).
Since \(|g(z)|\ge1\) on \(\D\), the function \(g\) is zero-free and
\[
        g^{-1}\in H^\infty,
        \qquad
        \|g^{-1}\|_{H^\infty}\le1.
\]
Thus \(U\) and \(V\) have the stated forms and are complex linear, and
\(VU=j_X\).
\end{proof}

\subsection{Lacunary estimates}

We use the classical Paley--Zygmund theorem for lacunary series.  A sequence of
positive integers \((n_k)\) is called lacunary if
\(\inf_k n_{k+1}/n_k>1\).

\begin{lemma}
\label{lem:paley}
Let \(1\le s<\infty\), and let \((n_k)\) be a sequence of positive
integers satisfying \(n_{k+1}\ge\lambda n_k\) for some \(\lambda>1\)
and all \(k\ge1\).  Then there exist constants
\(A_{s,\lambda},B_{s,\lambda}>0\) such that, for every finitely
supported scalar sequence \((a_k)\),
\[
        A_{s,\lambda}
        \left(\sum_k|a_k|^2\right)^{1/2}
        \le
        \left\|\sum_k a_kz^{n_k}\right\|_{H^s}
        \le
        B_{s,\lambda}
        \left(\sum_k|a_k|^2\right)^{1/2}.
\]
\end{lemma}

\begin{proof}
This is the classical Paley--Zygmund theorem for lacunary series; see
Zygmund~\cite[Vol.~I, Ch.~V, Theorem~8.20]{Zygmund2002}.
For the Hardy space formulation used here, see also
Duren~\cite[p.~104]{Duren1970} and the proof of
\cite[Proposition~3.2]{LNST2017}.
\end{proof}

%%%%%%%%%%%%%%%%%%%%%%%%%%%%%%%%%%%%%%%%%%%%%%%%%%%%%

\section{Pullback measures and compactness}
\label{sec:pullback}

We collect the Carleson and pullback measure tools used below,
summarize the classical boundedness and compactness criteria for
composition operators between Hardy spaces, and derive the compactness
consequences needed later.

%------------------------------------------------
\subsection{Closed-disk Carleson criteria}
\label{subsec:carleson}
%------------------------------------------------

Following Blasco--Jarchow~\cite{BlascoJarchow2005}, let
\(\mathcal I\) denote the family of half-open arcs \(I\subset\T\) with
\(0<|I|<1\), where \(|I|=m(I)\), and define
\[
        S(I)
        =
        \{r\zeta:\zeta\in I,\ 1-|I|\le r<1\}.
\]

Let \(\mu\) be a finite positive Borel measure on
\(\overline\D\), and write
\[
        \mu_\D=\mu|_\D,
        \qquad
        \mu_\T=\mu|_\T.
\]
Let  \(0<p,q<\infty\). For \(f\in H^p\), choose a Borel representative
of its boundary function \(f^*\) and put
\[
        f^\bullet(w)
        =
        \begin{cases}
        f(w), & w\in\D,\\[1mm]
        f^*(w), & w\in\T.
        \end{cases}
\]
Consider the formal map
\[
        J_\mu:H^p\longrightarrow L^q(\overline\D,\mu),
        \qquad
        J_\mu f=f^\bullet.
\]
We call \(\mu\) a \((p,q)\)-Carleson measure if this map is well
defined and bounded.

The closed-disk Carleson criteria needed below are as follows.

\begin{proposition}
\label{prop:carleson}
Let \(0<p\le q<\infty\), and let \(\mu\) be a finite positive Borel
measure on \(\overline\D\).

\begin{enumerate}[label=\textup{(\roman*)}]
\item If \(p<q\), then
\(J_\mu:H^p\to L^q(\mu)\) is well defined and bounded if and only if
\[
        \mu_\T=0
        \qquad\text{and}\qquad
        \sup_{I\in\mathcal I}
        \frac{\mu_\D(S(I))}{|I|^{q/p}}
        <\infty.
\]
If \(1\le p<q<\infty\), then \(J_\mu\) is compact if and only if
\[
        \mu_\T=0
        \qquad\text{and}\qquad
        \lim_{h\to0^+}
        \sup_{\substack{I\in\mathcal I\\ |I|\le h}}
        \frac{\mu_\D(S(I))}{|I|^{q/p}}
        =
        0.
\]

\item If \(p=q\), then
\(J_\mu:H^p\to L^p(\mu)\) is well defined and bounded if and only if
\[
        \mu_\T=F\,dm
        \quad\text{for some }F\in L^\infty(\T),
        \qquad
        \sup_{I\in\mathcal I}
        \frac{\mu_\D(S(I))}{|I|}
        <\infty.
\]
If \(1\le p<\infty\), then \(J_\mu\) is compact if and only if
\[
        \mu_\T=0
        \qquad\text{and}\qquad
        \lim_{h\to0^+}
        \sup_{\substack{I\in\mathcal I\\ |I|\le h}}
        \frac{\mu_\D(S(I))}{|I|}
        =
        0.
\]
\end{enumerate}
\end{proposition}

\begin{proof}
The boundedness assertions are
\cite[Theorems~2.5 and~2.6]{BlascoJarchow2005}, and the compactness
assertions are \cite[Theorem~3.4]{BlascoJarchow2005}.
\end{proof}

%------------------------------------------------
\subsection{Pullback measures and boundary contact}
\label{subsec:pullback}
%------------------------------------------------

Let \(\varphi:\D\to\D\) be analytic. By Fatou's theorem,
\(\varphi\) has radial boundary values \(\varphi^*\) almost everywhere
on \(\T\). After changing \(\varphi^*\) on an \(m\)-null set, we may
assume that
\(
        \varphi^*:\T\longrightarrow\overline{\D}
\)
is Borel measurable. The associated pullback measure on
\(\overline{\D}\) is defined by
\[
        \mu_\varphi(A)
        =
        m\{\zeta\in\T:\varphi^*(\zeta)\in A\},
        \qquad
        A\subset\overline{\D}\ \text{Borel}.
\]
Recall that
\(E_\varphi=\{\zeta\in\T:|\varphi^*(\zeta)|=1\}\). Then
\begin{equation}
\label{eq:boundary-mass}
        (\mu_\varphi)_\T(\T)
        =
        \mu_\varphi(\T)
        =
        m(E_\varphi).
\end{equation}

We first record the pullback identity on \(\overline{\D}\).

\begin{lemma}
\label{lem:pullback-id}
Let \(0<p,q<\infty\), and let \(\varphi:\D\to\D\) be analytic. For
every \(f\in H^p\), the function \(f^\bullet\) is well defined up to a
\(\mu_\varphi\)-null set, independently of the chosen Borel
representative of \(f^*\). Moreover,
\[
        \|C_\varphi f\|_{H^q}^q
        =
        \int_\T
        \bigl|f^\bullet(\varphi^*(\zeta))\bigr|^q\,dm(\zeta)
        =
        \int_{\overline{\D}}
        |f^\bullet(w)|^q\,d\mu_\varphi(w),
\]
where the common value \(+\infty\) is allowed. Consequently,
\(C_\varphi:H^p\to H^q\) is bounded if and only if
\[
        J_{\mu_\varphi}:H^p
        \longrightarrow
        L^q(\overline{\D},\mu_\varphi),
        \qquad
        J_{\mu_\varphi}f=f^\bullet,
\]
is bounded. If \(1\le p,q<\infty\), then
\(C_\varphi:H^p\to H^q\) is compact if and only if
\(J_{\mu_\varphi}\) is compact.
\end{lemma}

\begin{proof}
Blasco--Jarchow~\cite[Section~1]{BlascoJarchow2005} show that, for any
Borel representative of \(f^*\),
\[
        (f\circ\varphi)^*
        =
        f^\bullet\circ\varphi^*
        \quad\text{a.e. on }\T.
\]
Consequently, changing the Borel representative of \(f^*\) changes
\(f^\bullet\) only on a \(\mu_\varphi\)-null set. The boundary identity
and the definition of \(\mu_\varphi\) give the two integral identities.
The boundedness equivalence follows immediately.

Assume now that the equivalent boundedness conditions hold. Applying
the pullback identity to \(f-g\), we obtain
\[
        \|C_\varphi f-C_\varphi g\|_{H^q}
        =
        \|J_{\mu_\varphi}f-J_{\mu_\varphi}g\|
        _{L^q(\overline{\D},\mu_\varphi)},
        \qquad
        f,g\in H^p.
\]
The sequential criterion for compactness therefore shows that
\(C_\varphi\) is compact if and only if \(J_{\mu_\varphi}\) is compact.
\end{proof}

We next record the structure of the boundary part of the pullback
measure.

\begin{lemma}
\label{lem:boundary-part}
Let \(\varphi:\D\to\D\) be analytic, and write
\(\nu=(\mu_\varphi)_{\T}\). Then \(d\nu=h\,dm\) for some
nonnegative \(h\in L^\infty(\T)\). If \(m(E_\varphi)>0\), then there
exist a measurable set \(F\subset\T\) with \(m(F)>0\) and a constant
\(\delta>0\) such that \(h\ge\delta\) almost everywhere on \(F\).
\end{lemma}

\begin{proof}
Every analytic self-map of \(\D\) induces a bounded composition
operator on \(H^2\). By Lemma~\ref{lem:pullback-id},
\(J_{\mu_\varphi}:H^2\to L^2(\overline{\D},\mu_\varphi)\) is bounded.
Thus \(\mu_\varphi\) is a \((2,2)\)-Carleson measure, and
Proposition~\ref{prop:carleson}\textup{(ii)} gives
\(d\nu=h\,dm\) for some nonnegative \(h\in L^\infty(\T)\).

If \(m(E_\varphi)>0\), then
\eqref{eq:boundary-mass} gives
\[
        \int_{\T} h\,dm
        =
        \nu(\T)
        =
        m(E_\varphi)
        >
        0.
\]
Hence, for some \(\delta>0\), the set
\(F=\{\zeta\in\T:h(\zeta)\ge\delta\}\) has positive measure.
\end{proof}

%------------------------------------------------
\subsection{Classical criteria and the case
\texorpdfstring{\(p<q\)}{p<q}}
\label{subsec:smoothing}
%------------------------------------------------

We first summarize the classical boundedness and compactness criteria
for composition operators between Hardy spaces.

\begin{proposition}
\label{prop:classical}
Let \(1\le p,q<\infty\), and let \(\varphi:\D\to\D\) be analytic.
For \(w\in\D\setminus\{\varphi(0)\}\), let
\(
        N_\varphi(w)
        :=
        \sum\limits_{\varphi(z)=w}
        \log\frac{1}{|z|},
\)
where preimages are counted with multiplicity.

\begin{enumerate}[label=\textup{(\roman*)}]
\item If \(q<p\), then \(C_\varphi:H^p\to H^q\) is bounded, and
\[
        C_\varphi:H^p\to H^q\text{ is compact}
        \quad\Longleftrightarrow\quad
        m(E_\varphi)=0.
\]

\item If \(p=q\), then \(C_\varphi:H^p\to H^p\) is bounded, and
\[
        C_\varphi:H^p\to H^p\text{ is compact}
        \quad\Longleftrightarrow\quad
        N_\varphi(w)
        =
        o\!\left(\log\frac{1}{|w|}\right),
        \qquad |w|\to1.
\]

\item If \(p<q\), then
\[
        C_\varphi:H^p\to H^q\text{ is bounded}
        \quad\Longleftrightarrow\quad
        N_\varphi(w)
        =
        O\!\left(
        \left(\log\frac{1}{|w|}\right)^{q/p}
        \right),
        \qquad |w|\to1,
\]
and
\[
        C_\varphi:H^p\to H^q\text{ is compact}
        \quad\Longleftrightarrow\quad
        N_\varphi(w)
        =
        o\!\left(
        \left(\log\frac{1}{|w|}\right)^{q/p}
        \right),
        \qquad |w|\to1.
\]
\end{enumerate}
\end{proposition}

\begin{proof}
For \(q<p\), boundedness follows from Littlewood's subordination
principle and the inclusion \(H^p\subset H^q\). The compactness
criterion is due to Goebeler~\cite[Corollary~5, p.~390]{Goebeler2001}.

For \(p\le q\), the assertions follow from
\cite[Theorem~A, p.~311]{PerezGonzalezRattyaVukotic2007} by taking
\(\alpha=\beta=-1\). Indeed, with the convention \(A_{-1}^p=H^p\) used there,
\(N_{\varphi,1}=N_\varphi\) and
\(
\frac{q(2+\alpha)}{p}
=\frac{q}{p}.
\)
\end{proof}
The \(p=2\) case of part~\textup{(ii)} is Shapiro's classical criterion in terms of the Nevanlinna counting
function~\cite{Shapiro1987}.
The criteria above summarize the classical boundedness and compactness
theory for composition operators between Hardy spaces. In the
fixed-copy arguments below, we use their equivalent pullback-measure
formulations, which make explicit the boundary part of
\(\mu_\varphi\) and the relevant local Carleson estimates.

We next record the consequences needed in the range \(p<q\).

\begin{corollary}
\label{cor:boundary-obstruction}
Let \(1\le p<q<\infty\), and suppose that
\(C_\varphi:H^p\to H^q\) is bounded. Then
\[
        (\mu_\varphi)_{\T}=0,
        \qquad
        (\mu_\varphi)_{\D}(S(I))
        \lesssim |I|^{q/p},
        \qquad I\in\mathcal I.
\]
In particular, \(m(E_\varphi)=0\). Moreover,
\(C_\varphi:H^p\to H^p\) is compact, and
\(C_\varphi:H^r\to H^q\) is compact for every \(r>p\).
\end{corollary}

\begin{proof}
By Lemma~\ref{lem:pullback-id} and
Proposition~\ref{prop:carleson}\textup{(i)}, boundedness of
\(C_\varphi:H^p\to H^q\) gives the two assertions concerning
\(\mu_\varphi\). Equation~\eqref{eq:boundary-mass} then
gives \(m(E_\varphi)=0\).

Since \(q/p>1\),
\[
        \sup_{\substack{I\in\mathcal I\\ |I|\le h}}
        \frac{(\mu_\varphi)_{\D}(S(I))}{|I|}
        \lesssim
        h^{q/p-1}
        \longrightarrow0
        \qquad (h\to0^+).
\]
Proposition~\ref{prop:carleson}\textup{(ii)} and
Lemma~\ref{lem:pullback-id} now show that
\(C_\varphi:H^p\to H^p\) is compact.

Let \(p<r\le q\). Since
\[
        \sup_{\substack{I\in\mathcal I\\ |I|\le h}}
        \frac{(\mu_\varphi)_{\D}(S(I))}{|I|^{q/r}}
        \lesssim
        h^{q/p-q/r}
        \longrightarrow0
        \qquad (h\to0^+),
\]
Proposition~\ref{prop:carleson}, part~\textup{(i)} if
\(r<q\) and part~\textup{(ii)} if \(r=q\), together with
Lemma~\ref{lem:pullback-id}, shows that
\(C_\varphi:H^r\to H^q\) is compact.

Finally, if \(r>q\), let
\(\iota_{r,q}:H^r\hookrightarrow H^q\) denote the canonical inclusion.
The case \(r=q\) gives compactness of
\(C_\varphi:H^q\to H^q\). Since
\(C_\varphi|_{H^r}
=
\bigl(C_\varphi|_{H^q}\bigr)\circ\iota_{r,q}\),
it follows that \(C_\varphi:H^r\to H^q\) is compact.
\end{proof}

%-----------------------------------------------
%------------------------------------------------
\section{Localized analytic copies on sets of positive measure}
\label{sec:localization}

We prove Theorem~B in this section. First note that if
\(0<m(E)<1\) and \(1\le s\le p\), the restriction operator
\[
        R_E:H^p\longrightarrow L^s(E),
        \qquad
        R_Ef=f^*|_E,
\]
is not bounded below on \(H^p\). Indeed, for
\(0<\varepsilon<1\), let \(F_\varepsilon\) be an outer function with
boundary modulus \(\varepsilon\) on \(E\) and \(1\) on
\(\T\setminus E\).
Then
\[
        \|F_\varepsilon\|_{L^s(E)}
        =
        \varepsilon m(E)^{1/s},
        \qquad
        \|F_\varepsilon\|_{H^p}^p
        =
        \varepsilon^p m(E)+m(\T\setminus E).
\]
Consequently,
\[
        \frac{\|F_\varepsilon\|_{L^s(E)}}
             {\|F_\varepsilon\|_{H^p}}
        \longrightarrow0
        \qquad
        (\varepsilon\to0^+).
\]
Thus the desired lower estimate can hold only after restricting
\(R_E\) to a suitable subspace of \(H^p\).

The case \(r=2\) uses a dyadic lacunary tail construction. For \(p<r<2\), we combine
a localized realization of Kanter's stable integral with analytic
lifting through the Riesz projection. In both cases it suffices to
establish a lower estimate in \(L^1(E)\).  Indeed,   H\"older's inequality yields
\[
\|f\|_{L^s(E)}
\ge
m(E)^{1/s-1}\|f\|_{L^1(E)},\qquad
1\le s\le p.
\]

%------------------------------------------------
\subsection{Localized lacunary subspaces}
\label{subsec:lacunary}
%------------------------------------------------

The following lemma gives a lower \(L^2(E)\)-estimate for every
lacunary sequence after finitely many initial terms are removed.

\begin{lemma}
\label{lem:lacunary-set}
Let \(E\subset\T\) be measurable with \(m(E)>0\), and let \((n_k)\) be
a sequence of positive integers such that
\(n_{k+1}/n_k\ge\lambda>1\) for all \(k\ge1\). Then there exists
\(N=N(E,\lambda)\ge1\) such that
\begin{equation}
\label{eq:lacunary-set}
        \int_E
        \left|
        \sum_{k\ge N}a_k\zeta^{n_k}
        \right|^2\,dm(\zeta)
        \ge
        \frac{m(E)}{2}
        \sum_{k\ge N}|a_k|^2
\end{equation}
for every \((a_k)_{k\ge N}\in\ell^2\), where the series is understood
as its \(H^2\)-limit.
\end{lemma}

\begin{proof}
Let \(B_{4,\lambda}\) be the upper constant in
Lemma~\ref{lem:paley} for exponent \(4\), and choose
\(0<\varepsilon<m(E)/(2B_{4,\lambda}^2)\). Put \(h=\mathbf1_E-m(E)\). Then \(h\) is real-valued and
\(\widehat h(0)=0\). Since the trigonometric system is complete in
\(L^2(\T)\), the symmetric Fourier partial sums
\[
        Q_d(\zeta)
        =
        \sum_{0<|n|\le d}\widehat h(n)\zeta^n
\]
converge to \(h\) in \(L^2(\T)\). Hence, for some \(d\), \(Q:=Q_d\)
satisfies
\(\|h-Q\|_{L^2(\T)}<\varepsilon\). Moreover,
\(\widehat Q(0)=0\), and \(Q\) is real-valued because
\(\widehat h(-n)=\overline{\widehat h(n)}\). Set
\(P=m(E)+Q\). Then \(\widehat P(0)=m(E)\) and
\[\|\mathbf1_E-P\|_{L^2(\T)}<\varepsilon.\]

Since \(n_k\ge\lambda^{k-1}\), we may choose \(N\) so that
\((\lambda-1)\lambda^{N-1}>d\). If \(j>k\ge N\), then
\[
        n_j-n_k
        \ge
        n_{k+1}-n_k
        \ge
        (\lambda-1)n_k
        \ge
        (\lambda-1)\lambda^{k-1}
        >
        d.
\]
Thus \(|n_j-n_k|>d\) whenever \(j\ne k\) and \(j,k\ge N\).

Let \(f(\zeta)=\sum\limits_{k\ge N}b_k\zeta^{n_k}\), where \((b_k)\) is
finitely supported. Since \(P\) has degree at most \(d\) and
\(|n_j-n_k|>d\) for \(j\ne k\), it follows that
\(\widehat P(n_k-n_j)=0\) whenever \(j\ne k\). Hence
\[
        \int_\T P|f|^2\,dm
        =
        \sum_{j,k\ge N}
        b_j\overline{b_k}\widehat P(n_k-n_j)
        =
        m(E)\sum_{k\ge N}|b_k|^2.
\]
Moreover, Cauchy--Schwarz and
Lemma~\ref{lem:paley} give
\[
\begin{aligned}
        \left|
        \int_\T(\mathbf1_E-P)|f|^2\,dm
        \right|
        &\le
        \|\mathbf1_E-P\|_{L^2(\T)}
        \|f\|_{L^4(\T)}^2                                      \\
        &\le
        \varepsilon B_{4,\lambda}^2
        \sum_{k\ge N}|b_k|^2.
\end{aligned}
\]
Consequently,
\[
        \int_E|f|^2\,dm
        \ge
        \bigl(m(E)-\varepsilon B_{4,\lambda}^2\bigr)
        \sum_{k\ge N}|b_k|^2
        \ge
        \frac{m(E)}{2}
        \sum_{k\ge N}|b_k|^2.
\]
Now let \((a_k)_{k\ge N}\in\ell^2\), and put
\(
f_M(\zeta)
=
\sum\limits_{k=N}^M a_k\zeta^{n_k}.
\)
Then \(f_M\) converges in \(H^2\) to
\(
 f(\zeta)
=
\sum\limits_{k\ge N}a_k\zeta^{n_k}.
\)
In particular, \(f_M\to f\) in \(L^2(E)\). Applying the preceding
estimate to \(f_M\) and passing to the limit, we obtain
\[
        \int_E|f|^2\,dm
        \ge
        \frac{m(E)}{2}
        \sum_{k\ge N}|a_k|^2.
\]
This proves \eqref{eq:lacunary-set}.
\end{proof}

The following corollary proves the case \(r=2\) of Theorem~B. The
resulting lower bound depends on \(E\) only through \(m(E)\), although
the starting index \(N\) may depend on \(E\) itself.

\begin{corollary}
\label{cor:lacunary-tail}
Let \(1\le p<2\), let \(E\subset\T\) be measurable, and set
\(\delta=m(E)>0\). There exists an integer \(N=N(E)\ge1\) such that
\[
        M_{2,E}
        :=
        \overline{\operatorname{span}}^{\,H^p}
        \{z^{2^k}:k\ge N\}
\]
is isomorphic to \(\ell^2\), and hence to \(L^2(0,1)\). Moreover,
for every \(1\le s\le p\),
\begin{equation}
\label{eq:lacunary-tail}
        \|f\|_{L^s(E)}
        \ge
        D_p\delta^{\,1/2+1/s}\|f\|_{H^p},
        \qquad
        f\in M_{2,E},
\end{equation}
where
\(D_p=2^{-3/2}B_{4,2}^{-2}B_{p,2}^{-1}\).
\end{corollary}

\begin{proof}
Choose \(N\) as in Lemma~\ref{lem:lacunary-set} for the
dyadic sequence \(n_k=2^k\). Let
\(f=\sum\limits_{k\ge N}a_kz^{2^k}\) be a polynomial, and set
\(A=(\sum\limits_{k\ge N}|a_k|^2)^{1/2}\). Since
\(\|f\|_{L^4(E)}\le\|f\|_{H^4}\le B_{4,2}A\),
Lemma~\ref{lem:lacunary-set} and H\"older's inequality on \(E\)
give
\[
\begin{aligned}
        \left(\frac{\delta}{2}\right)^{1/2}A
        &\le
        \|f\|_{L^2(E)}
        \le
        \|f\|_{L^1(E)}^{1/3}\|f\|_{L^4(E)}^{2/3} \\
        &\le
        B_{4,2}^{\,2/3}\|f\|_{L^1(E)}^{1/3}A^{2/3}.
\end{aligned}
\]
It follows that
\(\|f\|_{L^1(E)}
\ge2^{-3/2}\delta^{3/2}B_{4,2}^{-2}A\).
Since Lemma~\ref{lem:paley} also gives
\(\|f\|_{H^p}\le B_{p,2}A\), we obtain
\[
        \|f\|_{L^1(E)}
        \ge
        D_p\delta^{3/2}\|f\|_{H^p}.
\]

For \(1\le s\le p\), H\"older's inequality on \(E\) gives
\[
        \|f\|_{L^s(E)}
        \ge
        \delta^{\,1/s-1}\|f\|_{L^1(E)}
        \ge
        D_p\delta^{\,1/2+1/s}\|f\|_{H^p}.
\]
Since \(s\le p\), H\"older's inequality shows that the boundary
restriction map \(H^p\to L^s(E)\) is bounded. As the lacunary
polynomials are dense in \(M_{2,E}\), passing to the limit gives
\eqref{eq:lacunary-tail} for every \(f\in M_{2,E}\).

Finally, Lemma~\ref{lem:paley} shows that
\((z^{2^k})_{k\ge N}\) is equivalent in \(H^p\) to the canonical basis
of \(\ell^2\). Therefore its closed linear span \(M_{2,E}\) is
isomorphic to \(\ell^2\), and hence to \(L^2(0,1)\).
\end{proof}

The tail estimate in Lemma~\ref{lem:lacunary-set} extends to
the full lacunary series for every \(0<s<\infty\).

\begin{proposition}
\label{prop:lacunary-local}
Let \(0<s<\infty\), let \((n_k)\) be a sequence of positive integers
satisfying \(n_{k+1}/n_k\ge\lambda>1\), and let \(E\subset\T\) be
measurable with \(m(E)>0\). Then there exists a constant
\(C=C(s,E,(n_k))>0\) such that
\begin{equation}
\label{eq:local-full}
        \left\|
        \sum_{k\ge1}a_k\zeta^{n_k}
        \right\|_{L^s(E)}
        \ge
        C\left(\sum_{k\ge1}|a_k|^2\right)^{1/2}
\end{equation}
for every \((a_k)_{k\ge1}\in\ell^2\), where the series on the left is
understood as the limit of its partial sums in \(H^s\).
\end{proposition}

\begin{proof}
Suppose first that \(s=2\).  Choose \(N\ge1\) as in
Lemma~\ref{lem:lacunary-set}, and put
\[
        X
        :=
        \overline{\operatorname{span}}^{\,H^2}
        \{z^{n_k}:k\ge1\},
        \qquad
        Y
        :=
        \overline{\operatorname{span}}^{\,H^2}
        \{z^{n_k}:k\ge N\},
\]
and
\[
        F
        :=
        \operatorname{span}
        \{z^{n_1},\ldots,z^{n_{N-1}}\},
\]
where \(F=\{0\}\) if \(N=1\). Then
\(
        X=F\oplus Y.
\)

Let
\[
        R_E:X\longrightarrow L^2(E),
        \qquad
        R_Ef=f^*|_E,
\]
be the restriction operator. Since the monomials are orthonormal in
\(H^2\), Lemma~\ref{lem:lacunary-set} gives
\[
        \|R_Ef\|_{L^2(E)}
        \ge
        \left(\frac{m(E)}{2}\right)^{1/2}
        \|f\|_{H^2},
        \qquad
        f\in Y.
\]
Hence \(R_E(Y)\) is closed in \(L^2(E)\). Moreover, \(R_E(F)\) is
finite-dimensional, and therefore
\[
        R_E(X)
        =
        R_E(F)+R_E(Y)
\]
is also closed in \(L^2(E)\).

We next show that \(R_E\) is injective. If \(f\in X\) and \(R_Ef=0\),
then \(f^*=0\) almost everywhere on \(E\). Since \(m(E)>0\), the
boundary uniqueness theorem gives \(f=0\); see
\cite[Ch.~II, Theorem~2.2]{Duren1970}. Thus
\(
        R_E:X\longrightarrow R_E(X)
\)
is a bounded bijection between Banach spaces. By the bounded inverse
theorem, there exists \(c_2>0\) such that
\[
        \|f\|_{L^2(E)}
        \ge
        c_2\|f\|_{H^2},
        \qquad
        f\in X.
\]
By the orthonormality of \((z^{n_k})_{k\ge1}\) in \(H^2\), this proves
\eqref{eq:local-full} for \(s=2\).

Suppose next that \(2<s<\infty\). For
\((a_k)_{k\ge1}\in\ell^2\), put
\(
        f(z)=\sum_{k\ge1}a_kz^{n_k}.
\)
By Lemma~\ref{lem:paley}, \(f\in H^s\). Since \(m(E)<\infty\),
H\"older's inequality gives
\[
        \|f\|_{L^2(E)}
        \le
        m(E)^{1/2-1/s}\|f\|_{L^s(E)}.
\]
Hence, by the case \(s=2\),
\[
\begin{aligned}
        \|f\|_{L^s(E)}
        &\ge
        m(E)^{1/s-1/2}\|f\|_{L^2(E)}\\
        &\ge
        c_2\,m(E)^{1/s-1/2}
        \left(\sum_{k\ge1}|a_k|^2\right)^{1/2}.
\end{aligned}
\]
Thus \eqref{eq:local-full} holds for \(2<s<\infty\).

Now let \(0<s<2\), and choose \(t>2\) and \(0<\theta<1\) so that
\(\frac{1}{2}=\frac{\theta}{s}+\frac{1-\theta}{t}\). For \((a_k)\in\ell^2\), put
\(f(z)=\sum_{k\ge1}a_kz^{n_k}\) and
\(A=(\sum_{k\ge1}|a_k|^2)^{1/2}\). By
Lemma~\ref{lem:paley}, the series defining \(f\) converges in
\(H^t\), and
\[
        \|f\|_{L^t(E)}
        \le
        \|f\|_{H^t}
        \le
        B_{t,\lambda}A.
\]
H\"older's inequality on \(E\), together with the case \(s=2\), gives
\[
\begin{aligned}
        c_2A
        &\le
        \|f\|_{L^2(E)} 
        \le
        \|f\|_{L^s(E)}^\theta
        \|f\|_{L^t(E)}^{1-\theta} \\
        &\le
        B_{t,\lambda}^{\,1-\theta}
        \|f\|_{L^s(E)}^\theta A^{1-\theta}.
\end{aligned}
\]
If \(A>0\), rearranging yields
\[
        \|f\|_{L^s(E)}
        \ge
        c_2^{1/\theta}
        B_{t,\lambda}^{-(1-\theta)/\theta}
        \left(\sum_{k\ge1}|a_k|^2\right)^{1/2}.
\]
The case \(A=0\) is immediate. This proves
\eqref{eq:local-full} for \(0<s<2\). 
\end{proof}

\begin{corollary}
\label{cor:lacunary-set}
Let \(1\le p<\infty\), and let \(E\subset\T\) be measurable with
\(m(E)>0\). Then
\[
        M_2
        :=
        \overline{\operatorname{span}}^{\,H^p}
        \{z^{2^k}:k\ge1\}
\]
is isomorphic to \(\ell^2\). Moreover, for every \(1\le s\le p\),
there exists \(c_{p,s,E}>0\) such that
\[
        \|f\|_{L^s(E)}
        \ge
        c_{p,s,E}\|f\|_{H^p},
        \qquad
        f\in M_2.
\]
\end{corollary}

\begin{proof}
By Lemma~\ref{lem:paley},
\((z^{2^k})_{k\ge1}\) is equivalent in \(H^p\) to the canonical basis
of \(\ell^2\). So \(M_2\simeq\ell^2\).  For every finite dyadic
polynomial \(f=\sum_k a_kz^{2^k}\), Proposition~
\ref{prop:lacunary-local} and
Lemma~\ref{lem:paley} give
\[
        \|f\|_{L^s(E)}
        \ge
        C_{s,E}\|(a_k)\|_{\ell^2}
        \ge
        C_{s,E}B_{p,2}^{-1}\|f\|_{H^p}.
\]
The conclusion follows for all \(f\in M_2\) by density, with
\(c_{p,s,E}=C_{s,E}B_{p,2}^{-1}\).
\end{proof}

\begin{remark}
Unlike the \(E\)-adapted space \(M_{2,E}\) introduced in
Corollary~\ref{cor:lacunary-tail}, the dyadic lacunary
space \(M_2\) defined above is independent of the set \(E\). The
corresponding lower bound constant may depend on \(E\) itself, rather
than only on \(m(E)\). We shall use this fixed space below in the
off-diagonal case \(q<p\) when \(p\ge2\).
\end{remark}

%------------------------------------------------
\subsection{Stable integrals and analytic lifting}
\label{subsec:stable-lift}
%------------------------------------------------

For a finite measure space \((X,\mathcal A,\mu)\), let
\(L^0(X;\R)\) denote the space of real-valued measurable functions
that are finite almost everywhere, modulo equality almost everywhere.

We first record the distribution-preserving transfer on \(L^0\)
induced by a measure algebra isomorphism.

\begin{lemma}
\label{lem:L0-transfer}
Let \((\Omega,\mathcal F,\mathbb P)\) be a probability space, let
\(\lambda\) denote Lebesgue measure on \((0,1)\), and let
\(\mathcal L\) be the Lebesgue \(\sigma\)-algebra. Suppose that there
is a measure-preserving Boolean \(\sigma\)-isomorphism
\[
        \Phi:
        \mathcal F/\mathcal N_{\mathbb P}
        \longrightarrow
        \mathcal L/\mathcal N_\lambda,
\]
where \(\mathcal N_{\mathbb P}\) and \(\mathcal N_\lambda\) are the
corresponding null ideals. Then there exists a real linear map
\[
        U:L^0(\Omega;\R)\longrightarrow L^0(0,1;\R)
\]
such that \(Uf\) and \(f\) have the same distribution for every
\(f\in L^0(\Omega;\R)\).
\end{lemma}

\begin{proof}

For a real-valued simple function
\(f=\sum_{j=1}^n a_j\mathbf1_{A_j}\), where the \(A_j\)'s form a
measurable partition modulo null sets, choose representatives
\(B_j\in\mathcal L\) of \(\Phi([A_j])\) and put
\(U_0f=\sum_{j=1}^n a_j\mathbf1_{B_j}\).
The choice of the representatives \(B_j\) does not affect \(U_0f\)
modulo null sets. Passing to common refinements shows that the
definition is also independent of the representation of \(f\), and
that \(U_0\) is real linear. Measure preservation shows that \(U_0f\)
and \(f\) have the same distribution.

For a probability measure \(\mu\), define
\(\rho_\mu(f,g)=\int |f-g|/(1+|f-g|)\,d\mu\).
This metric metrizes convergence in measure. Since \(U_0\) is real
linear, for simple functions \(f\) and \(g\) we have
\(U_0f-U_0g=U_0(f-g)\). Moreover, \(U_0(f-g)\) and \(f-g\) have the
same distribution, and hence
\(\rho_\lambda(U_0f,U_0g)=\rho_{\mathbb P}(f,g)\).
Thus \(U_0\) is an isometry on simple functions with respect to these
metrics. Since simple functions are dense in \(L^0(\Omega;\R)\) for
convergence in measure and \(L^0(0,1;\R)\) is complete with respect
to \(\rho_\lambda\), \(U_0\) extends uniquely to a real linear map
\(U:L^0(\Omega;\R)\to L^0(0,1;\R)\) which preserves the metric.

Let \(f_n\) be simple and \(f_n\to f\) in measure. Then
\(U_0f_n\to Uf\) in measure. For every \(\theta\in\R\), continuity of
\(t\mapsto e^{i\theta t}\) gives
\(e^{i\theta f_n}\to e^{i\theta f}\) and
\(e^{i\theta U_0f_n}\to e^{i\theta Uf}\) in measure.
Since these functions are uniformly bounded and the underlying
measures are finite, both convergences also hold in \(L^1\).
Therefore
\[
\begin{aligned}
        \int_0^1 e^{i\theta Uf}\,d\lambda
        &=
        \lim_{n\to\infty}
        \int_0^1 e^{i\theta U_0f_n}\,d\lambda \\
        &=
        \lim_{n\to\infty}
        \int_\Omega e^{i\theta f_n}\,d\mathbb P
        =
        \int_\Omega e^{i\theta f}\,d\mathbb P.
\end{aligned}
\]
Thus \(Uf\) and \(f\) have the same characteristic function, and
hence the same distribution.
\end{proof}

Kanter's construction provides a single stable integral map together
with an exact distributional identity. Using the preceding
\(L^0\)-transfer, we record the resulting identities for all lower
moments on \(\T\).

\begin{proposition}
\label{prop:stable-embed}
Let \(0<r\le2\). There exists a real linear map
\[
        S_r:L^r(0,1;\R)\longrightarrow L^0(\T;\R)
\]
such that
\begin{equation}
\label{eq:stable-moments}
        \|S_rx\|_{L^\tau(\T)}
        =
        \kappa_{r,\tau}\|x\|_{L^r(0,1)},
        \qquad
        x\in L^r(0,1;\R),\quad 0<\tau<r.
\end{equation}
Here \(Z_r\) is a symmetric \(r\)-stable random variable, normalized by
\(\mathbb E e^{i\theta Z_r}=e^{-|\theta|^r}\), and
\(\kappa_{r,\tau}:=(\mathbb E|Z_r|^\tau)^{1/\tau}\in(0,\infty)\).
If \(1<r\le2\), then
\(S_r(L^r(0,1;\R))\subset L^1(\T;\R)\).
\end{proposition}

\begin{proof}
Kanter's stable integral construction
\cite[Sections~4--5, pp.~405--406]{Kanter1973}
provides, on a probability space
\((\Omega,\mathcal F,\mathbb P)\), a real linear map
\(I_r:L^r(0,1;\R)\to L^0(\Omega;\R)\) such that
\[
        \mathbb E e^{i\theta I_rx}
        =
        \exp\left(
        -|\theta|^r\|x\|_{L^r(0,1)}^r
        \right),
        \qquad
        x\in L^r(0,1;\R),\quad \theta\in\R.
\]
This is formula~\textup{(***)} in
\cite[p.~406]{Kanter1973}, with \(q=r\), \(f=x\), and \(v=\theta\).
Taking \(x=\mathbf1_{(0,1)}\), we see that
\(I_r\mathbf1_{(0,1)}\) has characteristic function
\(e^{-|\theta|^r}\in L^1(\R)\). Hence Fourier inversion shows that
its distribution is absolutely continuous. Thus
\((\Omega,\mathcal F,\mathbb P)\) is nonatomic: if \(A\) were an
atom of positive measure, then \(I_r\mathbf1_{(0,1)}\) would be
constant almost everywhere on \(A\), so its distribution would have
a point mass, contradicting its absolute continuity.

Kanter notes that the underlying probability space is separable in
the measure theoretic sense \cite[p.~405]{Kanter1973}. Hence the
classical isomorphism theorem for separable nonatomic probability
measure algebras
\cite[Section~41, Theorem~C, p.~173]{Halmos1974} gives a
measure-preserving Boolean \(\sigma\)-isomorphism
\[
        \Phi:
        \mathcal F/\mathcal N_{\mathbb P}
        \longrightarrow
        \mathcal L/\mathcal N_\lambda.
\]
This is the measure algebra argument used in the proof of Kanter's
corollary; see \cite[p.~407]{Kanter1973}. By
Lemma~\ref{lem:L0-transfer}, there is a distribution-preserving real
linear map \(U:L^0(\Omega;\R)\to L^0(0,1;\R)\).

Using the measure-preserving parametrization
\(u\mapsto e^{2\pi iu}\) of \(\T\), define \(S_r\) by
\((S_rx)(e^{2\pi iu})=(UI_rx)(u)\) for almost every
\(u\in[0,1)\). Then
\(S_r:L^r(0,1;\R)\to L^0(\T;\R)\) is real linear.
Since \(U\) preserves distributions and the parametrization above
preserves measure, while uniqueness of characteristic functions gives
\(I_rx\overset{d}{=}\|x\|_{L^r(0,1)}Z_r\), we have
\[
        S_rx
        \overset{d}{=}
        UI_rx
        \overset{d}{=}
        I_rx
        \overset{d}{=}
        \|x\|_{L^r(0,1)}Z_r,
        \qquad
        x\in L^r(0,1;\R).
\]Kanter's moment argument
\cite[pp.~406--407]{Kanter1973} gives
\(0<\mathbb E|Z_r|^\tau<\infty\) for \(0<\tau<r\), and hence
\[
        \|S_rx\|_{L^\tau(\T)}
        =
        \kappa_{r,\tau}\|x\|_{L^r(0,1)}.
\]
This proves \eqref{eq:stable-moments}. If \(1<r\le2\), the last
assertion follows by taking \(\tau=1\).
\end{proof}

\begin{remark}
For completeness, we record the exact value of the constant
\(\kappa_{r,\tau}\) in Proposition~\ref{prop:stable-embed}.
For the normalization
\(\mathbb E e^{i\theta Z_r}=e^{-|\theta|^r}\), the integral
representation in the proof of
\cite[Theorem~6.4.16, p.~173]{AlbiacKalton2016}, together with
integration by parts and the beta and duplication formulas for the
Gamma function, gives
\[
        \mathbb E|Z_r|^\tau
        =
        \frac{
        2^\tau
        \Gamma\bigl((1+\tau)/2\bigr)
        \Gamma\bigl(1-\tau/r\bigr)
        }{
        \sqrt{\pi}\,
        \Gamma\bigl(1-\tau/2\bigr)
        },
        \qquad
        0<\tau<r<2.
\]
When \(r=2\), \(Z_2\) is a centered Gaussian random variable with
variance \(2\), and hence
\[
        \mathbb E|Z_2|^\tau
        =
        \frac{
        2^\tau\Gamma\bigl((1+\tau)/2\bigr)
        }{
        \sqrt{\pi}
        },
        \qquad
        \tau>0.
\]
Thus \(0<\kappa_{r,\tau}<\infty\) whenever
\(0<\tau<r\le2\).
\end{remark}

Let \(P_+\) denote the Riesz projection onto the nonnegative Fourier
modes.  The next lemma is an immediate consequence of the M.~Riesz theorem.
We include the proof to fix the normalization and record the symmetry used below.  
For an analytic function \(f\) on \(\D\), set
\[
        (\mathfrak cf)(z)
        =
        \overline{f(\overline z)},
        \qquad
        z\in\D.
\]
Then \(\mathfrak c\) is a conjugate-linear isometric involution on
\(H^u\) for every \(1\le u<\infty\).

\begin{lemma}
\label{lem:real-lift}
Let \(1<t<\infty\), and let \(u\in L^t(\T)\) be real-valued. Define
\[
        \mathcal Au=2P_+u-\widehat u(0).
\]
Then \(\mathcal Au\in H^t\),
\(\operatorname{Re}(\mathcal Au)=u\) almost everywhere on \(\T\), and
\begin{equation}
\label{eq:lift-norm}
        \|u\|_{L^t}
        \le
        \|\mathcal Au\|_{H^t}
        \le
        \bigl(2\|P_+\|_{L^t\to L^t}+1\bigr)\|u\|_{L^t}.
\end{equation}
If, in addition,
\(u(\overline\zeta)=u(\zeta)\) almost everywhere on \(\T\), then
\(\mathfrak c(\mathcal Au)=\mathcal Au\).
\end{lemma}

\begin{proof}
By the M.~Riesz theorem, \(P_+\) is bounded on \(L^t(\T)\), and hence
\(\mathcal Au\in H^t\). Since \(u\) is real-valued,
\(\widehat u(0)\in\mathbb R\) and
\(\widehat u(-n)=\overline{\widehat u(n)}\) for \(n\ge1\).
Comparison of Fourier coefficients therefore gives
\(\operatorname{Re}(\mathcal Au)=u\) almost everywhere on \(\T\).
Consequently,
\(
\|u\|_{L^t}\le\|\mathcal Au\|_{H^t}.
\)

Since normalized Haar measure has total mass one,
\[
        |\widehat u(0)|
        \le
        \|u\|_{L^1}
        \le
        \|u\|_{L^t}.
\]
Hence
\[
        \|\mathcal Au\|_{H^t}
        \le
        2\|P_+u\|_{L^t}+|\widehat u(0)|
        \le
        \bigl(2\|P_+\|_{L^t\to L^t}+1\bigr)\|u\|_{L^t},
\]
which proves \eqref{eq:lift-norm}.

Suppose now that
\(u(\overline\zeta)=u(\zeta)\) almost everywhere on \(\T\).
The reflection symmetry gives
\(\widehat u(-n)=\widehat u(n)\), while the real-valuedness of \(u\)
gives
\(\widehat u(-n)=\overline{\widehat u(n)}\).
Thus \(\widehat u(n)\in\mathbb R\) for every \(n\in\mathbb Z\).
All Taylor coefficients of \(\mathcal Au\) are therefore real, and
hence
\(\mathfrak c(\mathcal Au)=\mathcal Au\).
\end{proof}

The following lemma constructs a measure-preserving parametrization
of any measurable subset of \(\T\) with positive measure.

\begin{lemma}
\label{lem:mp-param}
Let \(A\subset\T\) be Lebesgue measurable with \(m(A)>0\), and set
\(\nu_A=\frac{m|_A}{m(A)}.
\)
Then there exists a measurable map
\(\psi_A:A\to\T\) such that
\[
        \nu_A\bigl(\psi_A^{-1}(D)\bigr)=m(D)
\]
for every Lebesgue measurable set \(D\subset\T\). Consequently, for
every \(1\le\tau<\infty\), the map
\[
        U_A:L^\tau(\T;\R)\to L^\tau(A,\nu_A;\R),
        \qquad
        U_Ah=h\circ\psi_A,
\]
is a well-defined real linear isometry.
\end{lemma}

\begin{proof}
Choose a Borel set \(A_0\subset\T\) such that
\( m(A\mathbin{\triangle}A_0)=0.
\)
Then \(m(A_0)=m(A)\). Let \(\lambda\) denote Lebesgue measure on
\([0,1)\), let \(\chi(t)=e^{2\pi it}\), and set
\[
        B=\chi^{-1}(A_0),
        \qquad
        \beta=\lambda(B)=m(A_0)=m(A)>0.
\]
Define
\[
        F(t)
        =
        \frac1\beta\int_0^t\mathbf1_B(s)\,ds,
        \qquad
        0\le t\le1.
\]
Then \(F\) is absolutely continuous and nondecreasing,
\(F(0)=0\), \(F(1)=1\), and
\(F'(t)=\mathbf1_B(t)/\beta\) almost everywhere on \([0,1]\).
Hence the change-of-variables formula for absolutely continuous
monotone functions \cite[p.~156]{Rudin1987} gives
\[
        \int_0^1 \Psi(F(t))F'(t)\,dt
        =
        \int_0^1 \Psi(u)\,du
\]
for every nonnegative Lebesgue measurable function
\(\Psi\) on \([0,1]\).

Since \(A_0\) and \(B\) are Borel and
\(\chi|_B:B\to A_0\) is a Borel isomorphism, the formula
\[
        \psi_0(e^{2\pi it})
        =
        e^{2\pi iF(t)},
        \qquad
        t\in B,
\]
defines a Borel map \(\psi_0:A_0\to\T\). If
\(D\subset\T\) is Borel, then, applying the preceding formula to
\(\Psi(u)=\mathbf1_D(e^{2\pi iu})\), we obtain
\[
\begin{aligned}
        \frac1{\beta}
        m\bigl(\psi_0^{-1}(D)\bigr)
        &=
        \frac1\beta
        \int_B
        \mathbf1_D\bigl(e^{2\pi iF(t)}\bigr)\,dt\\
        &=
        \int_0^1
        \mathbf1_D\bigl(e^{2\pi iF(t)}\bigr)F'(t)\,dt\\
        &=
        \int_0^1\mathbf1_D(e^{2\pi iu})\,du
        =
        m(D).
\end{aligned}
\]

Fix \(\zeta_0\in\T\), and define
\[
        \psi_A(\zeta)
        =
        \begin{cases}
        \psi_0(\zeta),
                & \zeta\in A\cap A_0,\\
        \zeta_0,
                & \zeta\in A\setminus A_0.
        \end{cases}
\]
Since \(A\setminus A_0\) is null, \(\psi_A:A\to\T\) is measurable.
Moreover, for every Borel set \(D\subset\T\),
\[
        \nu_A\bigl(\psi_A^{-1}(D)\bigr)
        =
        m(D),
\]
because \(A\mathbin{\triangle}A_0\) is null and
\(m(A)=m(A_0)=\beta\).

Now let \(D\subset\T\) be Lebesgue measurable. Choose Borel sets
\(D_0,N\subset\T\) such that
\[
        D\mathbin{\triangle}D_0\subset N,
        \qquad
        m(N)=0.
\]
Since
\(\nu_A(\psi_A^{-1}(N))=m(N)=0\), the sets
\(\psi_A^{-1}(D)\) and \(\psi_A^{-1}(D_0)\) differ by a
\(\nu_A\)-null set. Hence \(\psi_A^{-1}(D)\) is
\(\nu_A\)-measurable and
\[
        \nu_A\bigl(\psi_A^{-1}(D)\bigr)
        =
        m(D).
\]

If \(h_1=h_2\) \(m\)-almost everywhere, then
\(h_1\circ\psi_A=h_2\circ\psi_A\) \(\nu_A\)-almost everywhere.
Thus \(U_Ah=h\circ\psi_A\) is well defined on
\(L^\tau(\T;\R)\). Moreover,
\[
        \int_A \Phi\circ\psi_A\,d\nu_A
        =
        \int_\T \Phi\,dm
\]
for every nonnegative Lebesgue measurable function \(\Phi\) on
\(\T\). Taking \(\Phi=|h|^\tau\) gives
\[
        \|U_Ah\|_{L^\tau(A,\nu_A)}
        =
        \|h\|_{L^\tau(\T)},
        \qquad
        h\in L^\tau(\T;\R).
\]
Hence \(U_A\) is a real linear isometry.
\end{proof}
%------------------------------------------------
%------------------------------------------------
\subsection{Proof of the localization theorem}
\label{subsec:local-Lr}
%------------------------------------------------

We now combine the preceding constructions to prove Theorem~B.  The case \(r=2\) follows
from the lacunary construction, while the rest of the proof treats
\(p<r<2\).

\begin{proof}[Proof of Theorem~B]
Fix \(1\le p<r\le2\), and let \(E_0\subset\T\) be measurable with
\(\delta=m(E_0)>0\). If \(r=2\),
Corollary~\ref{cor:lacunary-tail} gives a closed subspace
\(M_{p,2,E_0}\subset H^p\), with
\(M_{p,2,E_0}\simeq\ell^2\simeq L^2(0,1)\), such that, for every
\(1\le s\le p\),
\[
        \|f\|_{L^s(E_0)}
        \ge
        D_p\delta^{1/2+1/s}\|f\|_{H^p},
        \qquad
        f\in M_{p,2,E_0}.
\]
Thus the conclusion holds with
\(c_{p,2,s}(\delta)=D_p\delta^{\,1/2+1/s}\).  
We henceforth assume that
\(p<r<2\).

We carry out the construction for an arbitrary \(t\in(p,r)\) and
set \(t=(p+r)/2\) at the end. 
We first find a subset \(G_0\subset E_0\) with
\(m(G_0)\ge\delta^2\) that is invariant under a map of the form
\(\zeta\mapsto\eta\overline{\zeta}\). After a suitable rotation, the
image of \(G_0\) is invariant under complex conjugation. For \(\eta\in\T\), set
\(\sigma_\eta(\zeta)=\eta\overline\zeta\), \(\zeta\in\T\). The map \(\sigma_\eta\) is a measure-preserving involution.  Hence
\[
        m\bigl(E_0\cap\sigma_\eta(E_0)\bigr)
        =
        \int_{E_0}
        \mathbf1_{E_0}(\eta\overline\zeta)\,dm(\zeta).
\]
Fubini's theorem and the rotation invariance of \(m\) give
\[
\begin{aligned}
        \int_\T m\bigl(E_0\cap\sigma_\eta(E_0)\bigr)\,dm(\eta)
        =
        \int_{E_0}\int_\T
        \mathbf1_{E_0}(\eta\overline\zeta)\,dm(\eta)\,dm(\zeta)
        =
        m(E_0)^2.
\end{aligned}
\]
Therefore there exists \(\eta_0\in\T\) such that
\(G_0=E_0\cap\sigma_{\eta_0}(E_0)\) satisfies \(m(G_0)\ge\delta^2\).
Since \(\sigma_{\eta_0}\circ\sigma_{\eta_0}=\operatorname{Id}\),
\(\sigma_{\eta_0}(G_0)
=\sigma_{\eta_0}(E_0)\cap E_0=G_0\).
Choose \(\omega\in\T\) with \(\omega^2=\eta_0\), and put
\(\widetilde E=\overline\omega E_0\) and
\(G=\overline\omega G_0\). Then \(m(\widetilde E)=\delta\), \(G\subset\widetilde E\),
\(m(G)\ge\delta^2\), and \(G\) is invariant under complex conjugation.
Indeed, if \(\zeta=\overline\omega\xi\in G\), then
\(\overline\zeta
=\omega\overline\xi
=\overline\omega\,\sigma_{\eta_0}(\xi)\in G\). The rotation
\((\mathcal R_\omega f)(z)=f(\overline\omega z)\)
is an isometry on \(H^p\), and
\[
        \|\mathcal R_\omega f\|_{L^s(E_0)}
        =
        \|f\|_{L^s(\widetilde E)},
        \qquad
        1\le s\le p.
\]
Thus any subspace obtained for \(\widetilde E\) is carried by
\(\mathcal R_\omega\) to a subspace for \(E_0\), with the same
constants.  It is therefore enough to work on \(\widetilde E\), where
\[
        G\subset\widetilde E,
        \qquad
        m(G)\ge\delta^2,
        \qquad
        \{\overline\zeta:\zeta\in G\}=G.
\]

We place the stable integral on one half of \(G\).  Let
\(G_+=G\cap\{\zeta\in\T:\operatorname{Im}\zeta>0\}\). Since \(G\) is invariant under complex conjugation,
\[
        G
        =
        G_+
        \cup
        \{\overline\zeta:\zeta\in G_+\}
        \cup
        \bigl(G\cap\{-1,1\}\bigr),
\]
and the three sets on the right are pairwise disjoint. Since conjugation preserves \(m\) and \(m(\{-1,1\})=0\),
we have \(m(G_+)=m(G)/2\). Equip \(G_+\) with the probability measure
\(d\nu=2m(G)^{-1}\,dm|_{G_+}\).

Since \(G_+\) is measurable and \(m(G_+)>0\),
Lemma~\ref{lem:mp-param} gives a measurable map \(\psi_{G_+}\colon G_+\to\T\) such that
\(\nu\bigl(\psi_{G_+}^{-1}(D)\bigr)=m(D)
\)
for every Lebesgue measurable set \(D\subset\T\).
Let \(S_r\) be the map given by
Proposition~\ref{prop:stable-embed}, and set
\[
        \widetilde S_r x
        =
        (S_r x)\circ\psi_{G_+},
        \qquad
        x\in L^r(0,1;\R).
\]
Then \(\widetilde S_r\) is a well-defined real linear map into
\(L^0(G_+,\nu;\R)\), and
\begin{equation}
\label{eq:stable-Gplus}
        \|\widetilde S_r x\|_{L^\tau(G_+,\nu)}
        =
        \|S_r x\|_{L^\tau(\T)}
        =
        \kappa_{r,\tau}\|x\|_{L^r(0,1)},
        \qquad
        1\le\tau\le t.
\end{equation}

For \(x\in L^r(0,1;\R)\), define a real-valued function \(u_x\) on
\(\T\) by
\[
        u_x(\zeta)
        =
        \begin{cases}
        (\widetilde S_rx)(\zeta),
                & \zeta\in G_+,\\[1mm]
        (\widetilde S_rx)(\overline\zeta),
                & \zeta\in
                  \{\overline\xi:\xi\in G_+\},\\[1mm]
        0,
                & \text{otherwise}.
        \end{cases}
\]
Then \(u_x\) is supported on \(G\) and satisfies
\(u_x(\overline\zeta)=u_x(\zeta)\) a.e. on \(\T\).  By
\eqref{eq:stable-Gplus},
\[
\begin{aligned}
        \|u_x\|_{L^\tau(\T)}^\tau
        &=
        2\int_{G_+}|\widetilde S_rx|^\tau\,dm =
        m(G)\int_{G_+}|\widetilde S_rx|^\tau\,d\nu\\
        &=
        m(G)\kappa_{r,\tau}^\tau
        \|x\|_{L^r(0,1)}^\tau.
\end{aligned}
\]
Put \(\gamma_\tau=m(G)^{1/\tau}\kappa_{r,\tau}\).  Then
\begin{equation}
\label{eq:sym-moments}
        \|u_x\|_{L^\tau(\T)}
        =
        \gamma_\tau\|x\|_{L^r(0,1)},
        \qquad
        1\le\tau\le t.
\end{equation}

Define \(T_{\R}x=\mathcal Au_x\). By
Lemma~\ref{lem:real-lift},
\(T_{\R}x\in H^t\subset H^p\) and
\(\operatorname{Re}(T_{\R}x)=u_x\). Put \(K_t=2\|P_+\|_{L^t\to L^t}+1\).  Since \(p<t\), it follows from
\eqref{eq:sym-moments} and
Lemma~\ref{lem:real-lift} that
\begin{equation}
\label{eq:real-Hp}
\begin{aligned}
        \gamma_p\|x\|_{L^r(0,1)}
        &=
        \|u_x\|_{L^p}
        \le
        \|T_{\R}x\|_{H^p}\\
        &\le
        \|T_{\R}x\|_{H^t}
        \le
        K_t\|u_x\|_{L^t}\\
        &=
        K_t\gamma_t\|x\|_{L^r(0,1)}.
\end{aligned}
\end{equation}
Since \(\operatorname{Re}(T_{\R}x)=u_x\) and \(u_x\) is supported
on \(G\), \eqref{eq:sym-moments} with \(\tau=1\) gives
\begin{equation}
\label{eq:real-local}
        \|T_{\R}x\|_{L^1(G)}
        \ge
        \|u_x\|_{L^1(G)}
        =
        \|u_x\|_{L^1(\T)}
        =
        \gamma_1\|x\|_{L^r(0,1)}.
\end{equation}
Since \(u_x\) is real-valued and
\(u_x(\overline\zeta)=u_x(\zeta)\), the symmetry assertion in
Lemma~\ref{lem:real-lift} gives
\begin{equation}
\label{eq:conj-invariance}
        \mathfrak c(T_{\R}x)
        =
        \mathfrak c(\mathcal Au_x)
        =
        \mathcal Au_x
        =
        T_{\R}x.
\end{equation}

We now complexify \(T_{\R}\).  For
\(x=a+ib\in L^r(0,1)\), with \(a\) and \(b\) real-valued, define
\(J_{p,r,\widetilde E}x=T_{\R}a+iT_{\R}b\). Since \(T_{\R}\) is real linear, this defines a complex linear operator
\(J_{p,r,\widetilde E}:L^r(0,1)\to H^p\).
Put \(F=J_{p,r,\widetilde E}x\).  It follows from
\eqref{eq:conj-invariance} that
\[
        T_{\R}a
        =
        \frac{F+\mathfrak cF}{2},
        \qquad
        T_{\R}b
        =
        \frac{F-\mathfrak cF}{2i}.
\]
The operator \(\mathfrak c\) is an isometry on \(H^p\).  Since \(G\)
is invariant under complex conjugation,
\[
\begin{aligned}
        \|\mathfrak cF\|_{L^1(G)}
        &=
        \int_G|F(\overline\zeta)|\,dm(\zeta)\\
        &=
        \int_G|F(\zeta)|\,dm(\zeta)
        =
        \|F\|_{L^1(G)}.
\end{aligned}
\]
Therefore
\[
\begin{aligned}
        \|F\|_{H^p}
        &\ge
        \max\bigl\{
        \|T_{\R}a\|_{H^p},
        \|T_{\R}b\|_{H^p}
        \bigr\},\\
        \|F\|_{L^1(G)}
        &\ge
        \max\bigl\{
        \|T_{\R}a\|_{L^1(G)},
        \|T_{\R}b\|_{L^1(G)}
        \bigr\}.
\end{aligned}
\]
Since \(r<2\), \(|a+ib|^r\le |a|^r+|b|^r\), and hence
\[
        \max\bigl\{
        \|a\|_{L^r(0,1)},
        \|b\|_{L^r(0,1)}
        \bigr\}
        \ge
        2^{-1/r}\|x\|_{L^r(0,1)}.
\]
Combining these estimates with
\eqref{eq:real-Hp} and
\eqref{eq:real-local}, we obtain
\begin{equation}
\label{eq:complex-Hp}
        2^{-1/r}\gamma_p\|x\|_{L^r(0,1)}
        \le
        \|J_{p,r,\widetilde E}x\|_{H^p}
        \le
        2K_t\gamma_t\|x\|_{L^r(0,1)},
\end{equation}
and,  since \(G\subset\widetilde E\) 
\begin{equation}
\label{eq:complex-local}
        \|J_{p,r,\widetilde E}x\|_{L^1(\widetilde E)}
        \ge
        2^{-1/r}\gamma_1\|x\|_{L^r(0,1)}.
\end{equation}
For the upper estimate in
\eqref{eq:complex-Hp}, we used
\(\|a\|_{L^r(0,1)},\|b\|_{L^r(0,1)}
\le\|x\|_{L^r(0,1)}\).
The two-sided estimate in
\eqref{eq:complex-Hp} shows that
\(J_{p,r,\widetilde E}\) is an isomorphic embedding.  Its range
\(\widetilde M_{p,r,\widetilde E}
=J_{p,r,\widetilde E}(L^r(0,1))\)
is therefore closed in \(H^p\), and
\(\widetilde M_{p,r,\widetilde E}\simeq L^r(0,1)\).

Let \(f=J_{p,r,\widetilde E}x\in
\widetilde M_{p,r,\widetilde E}\).  By
\eqref{eq:complex-Hp} and
\eqref{eq:complex-local},
\[
\begin{aligned}
        \|f\|_{L^1(\widetilde E)}
        &\ge
        2^{-1/r}\gamma_1\|x\|_{L^r(0,1)}\ge
        \frac{2^{-1-1/r}\gamma_1}
             {K_t\gamma_t}
        \|f\|_{H^p}\\
        &=
        \frac{2^{-1-1/r}\kappa_{r,1}}
             {K_t\kappa_{r,t}}
        m(G)^{1-1/t}\|f\|_{H^p}.
\end{aligned}
\]
Set
\[
        A_{p,r,t}
        :=
        \frac{2^{-1-1/r}\kappa_{r,1}}
             {K_t\kappa_{r,t}}.
\]
Since \(m(G)\ge\delta^2\) and \(t>1\),
\[
        \|f\|_{L^1(\widetilde E)}
        \ge
        A_{p,r,t}\,\delta^{\,2-2/t}\|f\|_{H^p}.
\]
For \(1\le s\le p\), H\"older's inequality gives
\[
        \|f\|_{L^s(\widetilde E)}
        \ge
        \delta^{\,1/s-1}\|f\|_{L^1(\widetilde E)}  \ge
        A_{p,r,t}\,
        \delta^{\,1+1/s-2/t}\|f\|_{H^p}.
\]

For Theorem~B, take \(t_0=(p+r)/2\), put
\(A_{p,r}=A_{p,r,t_0}\), and set
\(c_{p,r,s}(\delta)
=A_{p,r}\delta^{\,1+1/s-2/t_0}\).
The construction of the subspace is independent of \(s\), and the
preceding estimate holds simultaneously for every \(1\le s\le p\).

Finally, rotate back and put
\(M_{p,r,E_0}
=\mathcal R_\omega\widetilde M_{p,r,\widetilde E}\).
Then \(M_{p,r,E_0}\simeq L^r(0,1)\), and
\[
        \|f\|_{L^s(E_0)}
        \ge
        c_{p,r,s}(\delta)\|f\|_{H^p},
        \qquad
        1\le s\le p,\quad
        f\in M_{p,r,E_0}.
\]
This proves Theorem~B.
\end{proof}

 Since the argument is valid for every
\(p<t<r\), it also gives the more general estimate recorded in
Remark~\ref{rem:local-bound}.

\begin{remark}
\label{rem:local-bound}
Let \(1\le p<r<2\), let \(E\subset\T\) be measurable with
\(\delta=m(E)>0\), and let \(p<t<r\).  The preceding construction
gives a closed subspace
\[
        M_{p,r,t,E}\subset H^p,
        \qquad
        M_{p,r,t,E}\simeq L^r(0,1),
\]
and a constant \(A_{p,r,t}>0\), independent of \(E\) and \(s\), such
that
\[
        \|f\|_{L^s(E)}
        \ge
        A_{p,r,t}\,
        \delta^{\,1+1/s-2/t}
        \|f\|_{H^p},
        \qquad
        1\le s\le p,\quad f\in M_{p,r,t,E}.
\]

For \(r=2\), let \(M_{p,2,E}\) be the subspace given by
Corollary~\ref{cor:lacunary-tail}.  In this case,
\[
        \|f\|_{L^s(E)}
        \ge
        D_p\delta^{\,1/2+1/s}\|f\|_{H^p},
        \qquad
        1\le s\le p,\quad f\in M_{p,2,E},
\]
where \(D_p>0\) is independent of \(E\) and \(s\).  No optimality in the dependence on \(\delta\) is asserted.
\end{remark}

Since \(L^r(0,1)\) contains a closed subspace isometric to
\(\ell^r\), the same construction also gives a localized analytic copy
of \(\ell^r\).  For \(r=2\), such a copy is already provided by the
dyadic lacunary tail.

%%%%%%%%%%%%%%%%%%%%%%%%%%%%%%%%%%%%%%%%%%%%%
%%%%%%%%%%%%%%%%%%%%%%%%%%%%%%%%%%%%%%%%%%%%%
\section{\texorpdfstring{The case \(q<p\)}
{The case q < p}}
\label{sec:q-lt-p}

Throughout this section, \(1\le q<p<\infty\); in particular, \(p>1\).
Every analytic self-map \(\varphi:\D\to\D\) induces a bounded operator
\(C_\varphi:H^p\to H^q\), since \(C_\varphi\) is bounded on \(H^p\)
and the inclusion \(H^p\hookrightarrow H^q\) is bounded. We determine
the two fixed-copy exponent sets and prove that compactness and strict
singularity are equivalent in this case.

\begin{lemma}
\label{lem:boundary-lower}
Let \(1\le q<p<\infty\), and let \(\varphi:\D\to\D\) be analytic.
If \(C_\varphi:H^p\to H^q\) is noncompact, then there exist a measurable
set \(F\subset\T\) with \(m(F)>0\) and a constant \(c_0>0\) such that
\[
        \|C_\varphi f\|_{H^q}
        \ge
        c_0\|f\|_{L^q(F)},
        \qquad
        f\in H^p.
\]
\end{lemma}

\begin{proof}
By
Proposition~\ref{prop:classical}\textup{(i)},
\(m(E_\varphi)>0\).
Lemma~\ref{lem:boundary-part} gives
\(d(\mu_\varphi)_\T=h\,dm\) for some \(h\in L^\infty(\T)\), together
with a measurable set \(F\subset\T\) with \(m(F)>0\) and a constant
\(\eta>0\) such that \(h\ge\eta\) almost everywhere on \(F\). Hence
\[
        \|C_\varphi f\|_{H^q}^q
        =
        \int_{\overline\D}|f^\bullet|^q\,d\mu_\varphi
        \ge
        \int_\T |f^*|^q h\,dm
        \ge
        \eta\|f\|_{L^q(F)}^q,
        \qquad
        f\in H^p.
\]
The conclusion follows with \(c_0=\eta^{1/q}\).
\end{proof}

\begin{proposition}
\label{prop:stable-copies}
Let \(1\le q<p<r\le2\), and let \(\varphi:\D\to\D\) be analytic.
If \(C_\varphi:H^p\to H^q\) is noncompact, then it fixes copies of
both \(L^r(0,1)\) and \(\ell^r\).
\end{proposition}

\begin{proof}
By Lemma~\ref{lem:boundary-lower}, there exist a
measurable set \(F\subset\T\) with \(m(F)>0\) and a constant
\(c_0>0\) such that
\[
        \|C_\varphi f\|_{H^q}
        \ge
        c_0\|f\|_{L^q(F)},
        \qquad
        f\in H^p.
\]

Apply Theorem~B with \(E=F\) and \(s=q\). We obtain a closed subspace
\[
        M_{p,r,F}\subset H^p,
        \qquad
        M_{p,r,F}\simeq L^r(0,1),
\]
such that
\[
        \|f\|_{L^q(F)}
        \ge
        c_{p,r,q}(m(F))\|f\|_{H^p},
        \qquad
        f\in M_{p,r,F}.
\]
Hence
\[
        \|C_\varphi f\|_{H^q}
        \ge
        c_0c_{p,r,q}(m(F))\|f\|_{H^p},
        \qquad
        f\in M_{p,r,F}.
\]
Thus \(C_\varphi\) fixes a copy of \(L^r(0,1)\). Since
\(L^r(0,1)\) contains an isometric copy of \(\ell^r\),
\(C_\varphi\) also fixes a copy of \(\ell^r\).
\end{proof}

\begin{proposition}
\label{prop:ellp-singular}
Let \(1\le q<p<2\), and let \(\varphi:\D\to\D\) be analytic. Then
\(C_\varphi:H^p\to H^q\) fixes no copy of \(\ell^p\).
\end{proposition}

\begin{proof}
Suppose, to the contrary, that there exist a closed subspace
\(M\subset H^p\), with \(M\simeq\ell^p\), and a constant \(c>0\)
such that
\[
        \|C_\varphi f\|_{H^q}
        \ge
        c\|f\|_{H^p},
        \qquad
        f\in M.
\]
By Littlewood's subordination principle, there is a constant
\(B_{q,\varphi}>0\) such that
\[
        \|C_\varphi f\|_{H^q}
        \le
        B_{q,\varphi}\|f\|_{H^q},
        \qquad
        f\in H^q.
\]
Therefore, with \(A=B_{q,\varphi}/c\),
\[
        \|f\|_{H^p}
        \le
        A\|f\|_{H^q},
        \qquad
        f\in M.
\]

Choose \(0<\theta\leq 1\) so that
\(\frac1q=\theta+\frac{1-\theta}{p}\).
By H\"older's inequality,
\[
        \|f\|_{H^q}
        \le
        \|f\|_{H^1}^{\theta}
        \|f\|_{H^p}^{1-\theta},
        \qquad
        f\in H^p.
\]
Combining the preceding two estimates gives
\[
        \|f\|_{H^p}
        \le
        A^{1/\theta}\|f\|_{H^1},
        \qquad
        f\in M.
\]

Under the boundary-value identification,
\(H^p\) is a closed subspace of \(L^p(\T,m)\), so we may regard
\(M\) as a closed infinite-dimensional subspace of \(L^p(\T,m)\).
Moreover, \(M\simeq\ell^p\) and
\[
        \|f\|_{L^p(\T)}
        \le
        A^{1/\theta}\|f\|_{L^1(\T)},
        \qquad
        f\in M.
\]
This is condition~\textup{(2)} of
Rosenthal~\cite[Theorem~13, p.~370]{Rosenthal1973}, applied with
\(X=M\). 
The implication
\textup{(2)}\(\Rightarrow\)\textup{(3)} shows that \(M\) contains no
subspace isomorphic to \(\ell^p\), contradicting \(M\simeq\ell^p\).
\end{proof}
We can now determine both fixed-copy exponent sets.

\begin{theorem}
\label{thm:q-lt-p-class}
Let \(1\le q<p<\infty\), let \(1\le r<\infty\), and let
\(\varphi:\D\to\D\) be analytic.

\begin{enumerate}[label=\textup{(\roman*)}]
\item If \(1<p<2\), the following assertions are equivalent:
\begin{enumerate}[label=\textup{(\alph*)}]
\item \(C_\varphi:H^p\to H^q\) fixes a copy of \(L^r(0,1)\);
\item \(C_\varphi:H^p\to H^q\) fixes a copy of \(\ell^r\);
\item \(m(E_\varphi)>0\) and \(p<r\le2\).
\end{enumerate}

\item If \(2\le p<\infty\), the following assertions are equivalent:
\begin{enumerate}[label=\textup{(\alph*)}]
\item \(C_\varphi:H^p\to H^q\) fixes a copy of \(L^r(0,1)\);
\item \(C_\varphi:H^p\to H^q\) fixes a copy of \(\ell^r\);
\item \(m(E_\varphi)>0\) and \(r=2\).
\end{enumerate}
\end{enumerate}

Equivalently,
\[
        \operatorname{Fix}_{L}(C_\varphi)
        =
        \operatorname{Fix}_{\ell}(C_\varphi)
        =
        \begin{cases}
        \varnothing,
                & m(E_\varphi)=0,\\[1mm]
        (p,2],
                & m(E_\varphi)>0,\quad 1<p<2,\\[1mm]
        \{2\},
                & m(E_\varphi)>0,\quad 2\le p<\infty.
        \end{cases}
\]
\end{theorem}

\begin{proof}
Since \(L^r(0,1)\) contains an isometric copy of \(\ell^r\),
\textup{(a)} implies \textup{(b)} for both (i) and (ii).

{(i) \bf Case \(1<p<2\)}. \textup{(b)}\(\Rightarrow\)\textup{(c)}.  If \(C_\varphi\) fixes a copy of
\(\ell^r\), then it is noncompact. By
Proposition~\ref{prop:classical}\textup{(i)}, \(m(E_\varphi)>0\).
Proposition~\ref{prop:ambient-ellr} gives \(p\le r\le2\), while
Proposition~\ref{prop:ellp-singular} excludes \(r=p\). Hence
\(p<r\le2\), as desired.

\textup{(c)}\(\Rightarrow\)\textup{(a)}.  Suppose that \(m(E_\varphi)>0\) and \(p<r\le2\).
By Proposition~\ref{prop:classical}\textup{(i)},
\(C_\varphi:H^p\to H^q\) is noncompact. Hence
Proposition~\ref{prop:stable-copies} gives a
fixed copy of \(L^r(0,1)\).

{(ii) \bf Case \(p\ge2\)}.  \textup{(b)} \(\Rightarrow\)  \textup{(c)}.  Suppose that \(C_\varphi\) fixes a copy of
\(\ell^r\). Then it is noncompact, so
Proposition~\ref{prop:classical}\textup{(i)} gives \(m(E_\varphi)>0\).
Let \(M\subset H^p\) be a closed subspace isomorphic to \(\ell^r\) on
which \(C_\varphi\) is bounded below. Then \(C_\varphi(M)\) is a closed
subspace of \(H^q\) isomorphic to \(\ell^r\). We may therefore apply
Proposition~\ref{prop:ambient-ellr} to both \(M\subset H^p\) and
\(C_\varphi(M)\subset H^q\).   If \(p=2\), the proposition applied to \(H^2\) gives \(r=2\).
Suppose that \(p>2\). The source space gives \(r\in\{2,p\}\).
If \(q\le2\), the target space gives \(q\le r\le2\), and hence
\(r=2\). If \(q>2\), it gives \(r\in\{2,q\}\). Since \(q<p\), the
intersection of the two admissible sets is again \(\{2\}\).

\textup{(c)} \(\Rightarrow\) \textup{(a)}.  Suppose that \(m(E_\varphi)>0\) and \(r=2\).
By Lemma~\ref{lem:boundary-lower}, there exist a
measurable set \(F\subset\T\), with \(m(F)>0\), and a constant
\(c_0>0\) such that
\[
        \|C_\varphi f\|_{H^q}
        \ge
        c_0\|f\|_{L^q(F)},
        \qquad
        f\in H^p.
\]
Corollary~\ref{cor:lacunary-set}, applied with \(E=F\) and
\(s=q\), gives
\[
        \|f\|_{L^q(F)}
        \ge
        c_{p,q,F}\|f\|_{H^p},
        \qquad
        f\in M_2.
\]
Consequently,
\[
        \|C_\varphi f\|_{H^q}
        \ge
        c_0c_{p,q,F}\|f\|_{H^p},
        \qquad
        f\in M_2.
\]
Thus \(C_\varphi\) is bounded below on \(M_2\). Since
\(M_2\simeq\ell^2\simeq L^2(0,1)\), \(C_\varphi\) fixes copies of
both \(\ell^2\) and \(L^2(0,1)\). 
\end{proof}

\begin{theorem}
\label{thm:q-lt-p-identity}
Let \(1\le q<p<\infty\), and let \(\varphi:\D\to\D\) be analytic.
Then the following assertions are equivalent:
\begin{enumerate}[label=\textup{(\roman*)}]
\item \(C_\varphi:H^p\to H^q\) is compact;
\item \(m(E_\varphi)=0\);
\item \(C_\varphi:H^p\to H^q\) is strictly singular.
\end{enumerate}
\end{theorem}

\begin{proof}
By
Proposition~\ref{prop:classical}\textup{(i)},
compactness of \(C_\varphi:H^p\to H^q\) is equivalent to
\(m(E_\varphi)=0\). Every compact operator is strictly singular, so
\textup{(i)} implies \textup{(iii)}. If \textup{(ii)} fails, then
Theorem~\ref{thm:q-lt-p-class} shows that
\(C_\varphi\) fixes a copy of \(\ell^2\), and hence is not strictly
singular. Therefore \textup{(iii)} implies \textup{(ii)}.
\end{proof} \smallskip

%%%%%%%%%%%%%%%%%%%%%%%%%%%%%%%%%%%%%%%%%%
\section{\texorpdfstring{The case \(p<q\)}
{The case p < q}}
\label{sec:p-lt-q}

Throughout this section, we assume that \(1\le p<q<\infty\), \(\varphi:\D\to\D\) is analytic, and \(C_\varphi:H^p\to H^q\) is bounded. By
Corollary~\ref{cor:boundary-obstruction}, the operators
\[
        C_\varphi:H^p\to H^p
        \quad\text{and}\quad
        C_\varphi:H^t\to H^q
        \quad (p<t<\infty)
\]
are compact. The compactness of \(C_\varphi:H^t\to H^q\) has the
following multiplier consequence.

\begin{lemma}
\label{lem:mult-compact}
Let \(1\le p<q<\infty\), let \(t>p\), and let
\(\varphi:\D\to\D\) be analytic. Set
\(a=pt/(t-p)\).
If \(C_\varphi:H^p\to H^q\) is bounded and \(g\in H^a\), then
\(C_\varphi M_g:H^t\to H^q\) is compact.
\end{lemma}

\begin{proof}
Since \(1/p=1/a+1/t\), H\"older's inequality gives
\(\|gh\|_{H^p}\le\|g\|_{H^a}\|h\|_{H^t}\) for \(h\in H^t\).
Thus \(M_g:H^t\to H^p\) is bounded.

Since \(a<\infty\), analytic polynomials are dense in \(H^a\).
Choose polynomials \(P_n\) such that \(P_n\to g\) in \(H^a\).
Since \(P_n\in H^\infty\), each
\(M_{P_n}:H^t\to H^t\) is bounded. By
Corollary~\ref{cor:boundary-obstruction},
\(C_\varphi:H^t\to H^q\) is compact, and hence so is
\(C_\varphi M_{P_n}:H^t\to H^q\) for every \(n\).

Moreover,
\[
\begin{aligned}
\|C_\varphi M_g-C_\varphi M_{P_n}\|_{H^t\to H^q}
&\le
\|C_\varphi\|_{H^p\to H^q}
\|M_g-M_{P_n}\|_{H^t\to H^p}\\
&\le
\|C_\varphi\|_{H^p\to H^q}
\|g-P_n\|_{H^a}
\longrightarrow0.
\end{aligned}
\]
Since the compact operators from \(H^t\) to \(H^q\) form a
closed subspace of \(\mathcal L(H^t,H^q)\),
\(C_\varphi M_g:H^t\to H^q\) is compact.
\end{proof}

We first consider the case \(1\le p<2\).

\begin{theorem}
\label{thm:ss-below2}
Let \(1\le p<2\), \(p<q<\infty\), and let
\(\varphi:\D\to\D\) be analytic. If
\(C_\varphi:H^p\to H^q\) is bounded, then it is strictly singular.
\end{theorem}

\begin{proof}
Put \(T=C_\varphi\), and suppose that \(T\) is bounded below on an
infinite-dimensional closed subspace \(M\subset H^p\).

Assume first that \(p=1\). Since \(q>1\), the space \(H^q\) is
reflexive. The restriction \(T|_M\) is an isomorphism from \(M\) onto
the closed subspace \(T(M)\subset H^q\), and hence \(M\) is reflexive.
Lemma~\ref{lem:kp-H1} gives \(t>1\), an outer
function \(g\in H^a\), where \(\frac1 a=1-\frac 1 t\), and a bounded operator
\(U:M\to H^t\) such that \(f=gUf\) for every \(f\in M\).

Assume now that \(1<p<2\). We first show that \(M\) contains no closed
subspace isomorphic to \(\ell^p\). If \(Y\subset M\) were such a
subspace, then \(T(Y)\) would be a closed subspace of \(H^q\)
isomorphic to \(\ell^p\). If \(q\le2\),
Proposition~\ref{prop:ambient-ellr} would give \(q\le p\le2\),
contrary to \(p<q\). If \(q>2\), the same proposition would give
\(p\in\{2,q\}\), contrary to \(p<2\).

Lemma~\ref{lem:kp-below2} therefore gives
\(t\in(p,2)\), an outer function \(g\in H^a\), where
\(\frac1 a=\frac1 p-\frac1 t\), and a bounded operator \(U:M\to H^t\) such that
\(f=gUf\) for every \(f\in M\).

In either case,
\(T|_M=(C_\varphi M_g)U.
\)
By Lemma~\ref{lem:mult-compact},
\(C_\varphi M_g:H^t\to H^q\) is compact. Hence \(T|_M\) is compact,
which is impossible because a compact operator cannot be bounded
below on an infinite-dimensional space. Therefore \(T\) is strictly
singular.
\end{proof}

For \(p\ge2\), the following proposition shows that a non-strictly
singular operator is bounded below on a subspace isomorphic to
\(\ell^2\).

\begin{proposition}
\label{prop:ss-criterion}
Let \(2\le p<q<\infty\), and let
\(T:H^p\to H^q\) be a bounded operator. Then \(T\) is not strictly
singular if and only if it fixes a copy of \(\ell^2\).
\end{proposition}

\begin{proof}
If \(T\) fixes a copy of \(\ell^2\), then it is not strictly singular.
Conversely, suppose that \(T\) is bounded below on an
infinite-dimensional closed subspace \(M\subset H^p\).

If \(p=2\), then \(M\) is an infinite-dimensional Hilbert space and
hence isomorphic to \(\ell^2\). Thus \(T\) fixes a copy of \(\ell^2\).

Suppose that \(p>2\). Under the boundary-value identification,
\(M\) is a closed subspace of \(L^p(\T)\). By
Lemma~\ref{lem:kp-subspace}, \(M\) contains a closed
subspace \(Y\) isomorphic either to \(\ell^2\) or to \(\ell^p\). If
\(Y\) is isomorphic to \(\ell^2\), the conclusion follows.

Suppose instead that \(Y\) is isomorphic to \(\ell^p\). Since \(T\)
is bounded below on \(M\), the space \(T(Y)\) is closed in \(H^q\), and
\(T|_Y\colon Y\to T(Y)\) is an isomorphism.
Thus \(H^q\) contains a closed subspace isomorphic
to \(\ell^p\). Proposition~\ref{prop:ambient-ellr}, applied with
\(r=p\) and \(s=q\), then gives \(p\in\{2,q\}\), contrary to
\(2<p<q\). Therefore \(Y\) is isomorphic to \(\ell^2\), and \(T\)
fixes a copy of \(\ell^2\).
\end{proof}

It remains to consider \(2\le p<q\).

\begin{theorem}
\label{thm:ss-above2}
Let \(2\le p<q<\infty\), and let \(\varphi:\D\to\D\) be analytic.
If \(C_\varphi:H^p\to H^q\) is bounded, then it is strictly singular.
\end{theorem}

\begin{proof}
Suppose, to the contrary, that \(C_\varphi:H^p\to H^q\) is not
strictly singular. By Proposition~\ref{prop:ss-criterion}, there is a
closed subspace \(M\subset H^p\), isomorphic to \(\ell^2\), such that
\(C_\varphi\) is bounded below on \(M\). Put \(N=C_\varphi(M)\).
Then \(N\) is closed in \(H^q\), and
\(C_\varphi|_M:M\to N\) is an isomorphism. In particular,
\(N\simeq\ell^2\).

Under the boundary-value identification, \(N\) is a closed subspace of
\(L^q(\T)\). Lemma~\ref{lem:kp-hilbert} therefore
gives a constant \(A>0\) such that
\[
        \|g\|_{H^q}
        =
        \|g\|_{L^q(\T)}
        \le
        A\|g\|_{L^2(\T)}
        =
        A\|g\|_{H^2},
        \qquad
        g\in N.
\]
Since \(q>p\), we have \(N\subset H^q\subset H^p\). Moreover,
\(p\ge2\) and \(m(\T)=1\), so
\(\|g\|_{H^2}\le\|g\|_{H^p}\) for \(g\in H^p\). Consequently,
\[
        \|g\|_{H^q}
        \le
        A\|g\|_{H^p},
        \qquad
        g\in N.
\]

By Corollary~\ref{cor:boundary-obstruction},
\(C_\varphi:H^p\to H^p\) is compact. Let \((f_n)\) be a bounded
sequence in \(M\). After passing to a subsequence, we may assume that
\((C_\varphi f_n)\) is Cauchy in \(H^p\). Since
\(C_\varphi f_n-C_\varphi f_m\in N\), the preceding estimate gives
\[
        \|C_\varphi f_n-C_\varphi f_m\|_{H^q}
        \le
        A\|C_\varphi f_n-C_\varphi f_m\|_{H^p}
        \longrightarrow0
        \qquad (n,m\to\infty).
\]
Thus \((C_\varphi f_n)\) converges in \(H^q\), and hence
\(C_\varphi|_M:M\to H^q\) is compact. This is impossible because
\(C_\varphi|_M\) is bounded below and \(M\) is infinite-dimensional.
Therefore \(C_\varphi:H^p\to H^q\) is strictly singular.
\end{proof}

Combining the two source ranges gives the fixed-copy classification when \(p<q\).

\begin{corollary}
\label{cor:p-lt-q}
Let \(1\le p<q<\infty\), and let \(\varphi:\D\to\D\) be analytic.
If \(C_\varphi:H^p\to H^q\) is bounded, then \(C_\varphi\) is strictly
singular and
\[
        \operatorname{Fix}_{\ell}(C_\varphi)
        =
        \operatorname{Fix}_{L}(C_\varphi)
        =
        \varnothing.
\]
Consequently,
\(\CS(H^p,H^q)=\CB(H^p,H^q)\).
\end{corollary}

\begin{proof}
The strict singularity follows from
Theorems~\ref{thm:ss-below2} and
\ref{thm:ss-above2}.  If an operator fixes a copy of an
infinite-dimensional Banach space, then it is bounded below on an
infinite-dimensional subspace.  Hence a strictly singular operator fixes
no copy of \(\ell^r\) or \(L^r(0,1)\), and the two exponent sets are
empty.  Since every bounded composition operator \(H^p\to H^q\) with
\(p<q\) is strictly singular,  the asserted class identity follows.
\end{proof}

\begin{proof}[Proof of Theorem~A]
Assume first that \(q<p\). By Littlewood's subordination principle,
\(C_\varphi:H^p\to H^p\) is bounded. Since \(m(\T)=1\), the inclusion
\(H^p\hookrightarrow H^q\) is bounded, and hence
\(C_\varphi:H^p\to H^q\) is bounded. The first three rows of the table
follow from
Theorem~\ref{thm:q-lt-p-class}.
Moreover, Theorem~\ref{thm:q-lt-p-identity} gives
\[
C_\varphi:H^p\to H^q\text{ is compact}
\Longleftrightarrow
m(E_\varphi)=0
\Longleftrightarrow
C_\varphi\text{ is strictly singular}.
\]

Assume now that \(p<q\). By
Corollary~\ref{cor:boundary-obstruction}, boundedness implies
\(m(E_\varphi)=0\). Corollary~\ref{cor:p-lt-q}
shows that \(C_\varphi:H^p\to H^q\) is strictly singular and that
\(\operatorname{Fix}_{\ell}(C_\varphi)
=\operatorname{Fix}_{L}(C_\varphi)=\varnothing\).
The proof of Theorem~A is complete. 
\end{proof}
\smallskip

\section{The diagonal classification}
\label{sec:diag}

We now determine the fixed-copy exponent sets for composition operators
on \(H^p\). Laitila, Nieminen, Saksman, and Tylli~\cite{LNST2017}
determined when \(C_\varphi\) fixes copies of \(\ell^p\) and
\(\ell^2\), and, for \(1<p<\infty\), \(p\ne2\), when it fixes a copy
of \(L^p(0,1)\). When \(m(E_\varphi)>0\) and \(1\le p<2\), the
boundary lower estimate, combined with Theorem~B, yields fixed copies
of \(L^r(0,1)\), and hence of \(\ell^r\), for \(p<r<2\).
Proposition~\ref{prop:ambient-ellr} restricts the possible sequence
exponents, and hence also the possible function space exponents. The
remaining cases follow from the multiplier factorizations in
Section~\ref{sec:prelim}, the \(\ell^2\)-subspaces of
\(L^r(0,1)\), and the separate endpoint argument for \(p=1\) given
below.
%------------------------------------------------

\subsection{Zero boundary contact}
%------------------------------------------------

We first record the compactness consequence of
\(m(E_\varphi)=0\) that will be used with the multiplier
factorizations.

\begin{lemma}
\label{lem:diag-compact}
Let \(1\le p<t<\infty\), and let \(\varphi:\D\to\D\) be analytic
with \(m(E_\varphi)=0\). Set \(a=\frac{pt}{t-p}\), so that
\(\frac1 p=\frac1 a+\frac1 t\). If \(g\in H^a\), then
\(C_\varphi M_g:H^t\to H^p\) is compact.
\end{lemma}

\begin{proof}
By Theorem~\ref{thm:q-lt-p-identity},
\(C_\varphi:H^t\to H^p\) is compact. Choose analytic polynomials
\(P_n\) such that \(P_n\to g\) in \(H^a\). Since
\(M_{P_n}:H^t\to H^t\) is bounded, each operator
\(C_\varphi M_{P_n}:H^t\to H^p\) is compact. Moreover, H\"older's
inequality gives
\[
        \|C_\varphi M_g-C_\varphi M_{P_n}\|_{H^t\to H^p}
        \le
        \|C_\varphi\|_{H^p\to H^p}
        \|g-P_n\|_{H^a}
       \longrightarrow 0 \qquad (n\to\infty).
\]
Thus \(C_\varphi M_g\) is compact.
\end{proof}
Combining the preceding lemma with the two multiplier factorizations
gives the compactness statements needed below.

\begin{proposition}
\label{prop:diag-compact}
Let \(\varphi:\D\to\D\) be analytic, and assume that
\(m(E_\varphi)=0\).
\begin{enumerate}[label=\textup{(\roman*)}]
\item If \(X\) is a reflexive closed subspace of \(H^1\), then
\(C_\varphi|_X:X\to H^1\) is compact.
\item If \(1<p<2\) and \(X\) is a closed subspace of \(H^p\)
containing no copy of \(\ell^p\), then
\(C_\varphi|_X:X\to H^p\) is compact.
\end{enumerate}
\end{proposition}

\begin{proof}
For \textup{(i)},
Lemma~\ref{lem:kp-H1} gives \(t>1\),
\(g\in H^a\), where \(\frac1 a=1-\frac1 t\), and a bounded operator
\(U:X\to H^t\) such that \(M_gU=j_X\), where
\(j_X:X\to H^1\) denotes the natural inclusion. By
Lemma~\ref{lem:diag-compact},
\(C_\varphi M_g:H^t\to H^1\) is compact. Since \(U\) is bounded and
\(C_\varphi|_X=C_\varphi j_X=(C_\varphi M_g)U\),
it follows that \(C_\varphi|_X:X\to H^1\) is compact.

For \textup{(ii)},
Lemma~\ref{lem:kp-below2} gives \(t\in(p,2)\),
\(g\in H^a\), where \(\frac1 a=\frac1 p-\frac1 t\), and a bounded operator
\(U:X\to H^t\) such that \(M_gU=j_X\), where
\(j_X:X\to H^p\) denotes the natural inclusion. By
Lemma~\ref{lem:diag-compact},
\(C_\varphi M_g:H^t\to H^p\) is compact. Since \(U\) is bounded and
\(C_\varphi|_X=C_\varphi j_X=(C_\varphi M_g)U\),
we conclude that \(C_\varphi|_X:X\to H^p\) is compact.
\end{proof}
%------------------------------------------------
\subsection{Proof of Theorem~C}
%------------------------------------------------

We combine
Proposition~\ref{prop:diag-compact}
with the known fixed-copy results for \(\ell^p\), \(\ell^2\), and
\(L^p(0,1)\).

\begin{proof}
We use the following results of
Laitila--Nieminen--Saksman--Tylli.
\begin{enumerate}[label=\textup{(LNST\arabic*)},leftmargin=*]
\item For every \(1\le p<\infty\), the operator
\(C_\varphi:H^p\to H^p\) is noncompact if and only if it fixes a
copy of \(\ell^p\)
\cite[Theorem~1.2]{LNST2017}.

\item For \(1\le p<\infty\), \(p\ne2\), the operator
\(C_\varphi:H^p\to H^p\) fixes a copy of \(\ell^2\) if and only if
\(m(E_\varphi)>0\)
\cite[Theorem~1.4]{LNST2017}.

\item For \(1<p<\infty\), \(p\ne2\), the operator
\(C_\varphi:H^p\to H^p\) fixes a copy of \(L^p(0,1)\) if and only if
\(m(E_\varphi)>0\)
\cite[Theorem~1.5]{LNST2017}.
\end{enumerate}

Recall that
\(\operatorname{Fix}_{L}(C_\varphi)
\subseteq\operatorname{Fix}_{\ell}(C_\varphi)\).
If \(C_\varphi\) is compact, then it fixes no copy of an
infinite-dimensional Banach space, and hence
\[
        \operatorname{Fix}_{\ell}(C_\varphi)
        =
        \operatorname{Fix}_{L}(C_\varphi)
        =
        \varnothing.
\]
For the remainder of the proof, assume that \(C_\varphi\) is
noncompact.

Suppose first that \(p=2\). Since
\(L^2(0,1)\simeq\ell^2\), \textup{(LNST1)} gives
\(2\in\operatorname{Fix}_{L}(C_\varphi)
\subseteq\operatorname{Fix}_{\ell}(C_\varphi)\).
On the other hand,
Proposition~\ref{prop:ambient-ellr} gives
\(\operatorname{Fix}_{\ell}(C_\varphi)\subseteq\{2\}\).
Therefore
\[
        \operatorname{Fix}_{\ell}(C_\varphi)
        =
        \operatorname{Fix}_{L}(C_\varphi)
        =
        \{2\}.
\]

Assume next that \(p\ne2\) and \(m(E_\varphi)=0\).
By \textup{(LNST1)},
\(p\in\operatorname{Fix}_{\ell}(C_\varphi)\).

Let \(p=1\). Suppose that
\(r\in\operatorname{Fix}_{\ell}(C_\varphi)\) for some \(r>1\),
and choose \(M\simeq\ell^r\) such that \(C_\varphi|_M\) is bounded
below. Since \(\ell^r\) is reflexive for \(1<r<\infty\), so is \(M\).
Proposition~\ref{prop:diag-compact}%
\textup{(i)} then implies that \(C_\varphi|_M\) is compact, a
contradiction. Thus
\[
        \operatorname{Fix}_{\ell}(C_\varphi)=\{1\}.
\]

Let \(1<p<2\). By \textup{(LNST1)} and
Proposition~\ref{prop:ambient-ellr},
\[
        p
        \in
        \operatorname{Fix}_{\ell}(C_\varphi)
        \subseteq
        [p,2].
\]
Suppose that
\(r\in\operatorname{Fix}_{\ell}(C_\varphi)\) for some \(p<r\le2\),
and choose \(M\simeq\ell^r\) such that \(C_\varphi|_M\) is bounded
below. Since \(p<r\),
Lemma~\ref{lem:ellr-rigidity} shows that \(M\) contains no
copy of \(\ell^p\).
Proposition~\ref{prop:diag-compact}%
\textup{(ii)} then implies that \(C_\varphi|_M\) is compact, again a
contradiction. Therefore
\[
        \operatorname{Fix}_{\ell}(C_\varphi)=\{p\}.
\]

Let \(p>2\). By \textup{(LNST1)} and
Proposition~\ref{prop:ambient-ellr},
\(p\in\operatorname{Fix}_{\ell}(C_\varphi)\subseteq\{2,p\}\).
Since \(m(E_\varphi)=0\), \textup{(LNST2)} excludes \(2\). Hence
\[
        \operatorname{Fix}_{\ell}(C_\varphi)=\{p\}.
\]

It remains to determine
\(\operatorname{Fix}_{L}(C_\varphi)\) when \(p\ne2\) and
\(m(E_\varphi)=0\). Suppose that
\(r\in\operatorname{Fix}_{L}(C_\varphi)\) for some
\(1\le r<\infty\). Choose \(M\simeq L^r(0,1)\) such that
\(C_\varphi|_M\) is bounded below, and let
\(J:L^r(0,1)\to M\) be an isomorphism.

Let \((r_n)\) be the Rademacher functions on \((0,1)\), and put
\[
        R
        =
        \overline{\operatorname{span}}^{\,L^r(0,1)}
        \{r_n:n\in\mathbb N\}.
\]
By Khintchine's inequalities, \(R\simeq\ell^2\). Hence \(J(R)\)
is a closed subspace of \(M\) isomorphic to \(\ell^2\), and
\(C_\varphi|_{J(R)}\) is bounded below. Thus
\(2\in\operatorname{Fix}_{\ell}(C_\varphi)\), contradicting
\textup{(LNST2)}. Therefore
\[
        \operatorname{Fix}_{L}(C_\varphi)=\varnothing.
\]

We now consider \(p\ne2\) and \(m(E_\varphi)>0\).

Suppose first that \(p>2\). Since
\(\ell^2\simeq L^2(0,1)\), \textup{(LNST2)} gives
\(2\in\operatorname{Fix}_{L}(C_\varphi)\), while
\textup{(LNST3)} gives
\(p\in\operatorname{Fix}_{L}(C_\varphi)\). Together with
Proposition~\ref{prop:ambient-ellr}, this yields
\[
        \{2,p\}
        \subseteq
        \operatorname{Fix}_{L}(C_\varphi)
        \subseteq
        \operatorname{Fix}_{\ell}(C_\varphi)
        \subseteq
        \{2,p\}.
\]
Consequently,
\[
        \operatorname{Fix}_{\ell}(C_\varphi)
        =
        \operatorname{Fix}_{L}(C_\varphi)
        =
        \{2,p\}.
\]

It remains to treat \(1\le p<2\). By
Lemma~\ref{lem:boundary-part}, there exist a measurable set
\(F\subset\T\) with \(m(F)>0\), a function
\(h\in L^\infty(\T)\), and a constant \(\eta>0\) such that
\(d(\mu_\varphi)_\T=h\,dm\) and \(h\ge\eta\) almost everywhere on
\(F\). The pullback identity gives
\begin{equation}
\label{eq:diag-lower}
        \|C_\varphi f\|_{H^p}^p
        \ge
        \eta\|f\|_{L^p(F)}^p,
        \qquad
        f\in H^p.
\end{equation}

Let \(p=1\). For each \(1<r\le2\), Theorem~B, applied with
\(E=F\) and \(s=1\), together with
\eqref{eq:diag-lower}, gives
\(r\in\operatorname{Fix}_{L}(C_\varphi)\). Thus
\[
        (1,2]
        \subseteq
        \operatorname{Fix}_{L}(C_\varphi)
        \subseteq
        \operatorname{Fix}_{\ell}(C_\varphi).
\]
By \textup{(LNST1)},
\(1\in\operatorname{Fix}_{\ell}(C_\varphi)\), while
Proposition~\ref{prop:ambient-ellr} gives
\(\operatorname{Fix}_{\ell}(C_\varphi)\subseteq[1,2]\).
Consequently,
\[
        \operatorname{Fix}_{\ell}(C_\varphi)=[1,2].
\]

It remains to exclude the endpoint \(1\) from
\(\operatorname{Fix}_{L}(C_\varphi)\).
If \(1\in\operatorname{Fix}_{L}(C_\varphi)\), then \(H^1\) would
contain a closed subspace isomorphic to \(L^1(0,1)\), which is
impossible; see \cite[p.~262]{KwapienPelczynski1976}. Hence
\(1\notin\operatorname{Fix}_{L}(C_\varphi)\), and therefore
\[
        \operatorname{Fix}_{L}(C_\varphi)=(1,2].
\]

Finally, let \(1<p<2\). By \textup{(LNST3)},
\(p\in\operatorname{Fix}_{L}(C_\varphi)\). For every \(p<r\le2\),
Theorem~B, applied with \(E=F\) and \(s=p\), together with
\eqref{eq:diag-lower}, gives
\(r\in\operatorname{Fix}_{L}(C_\varphi)\). Hence
\[
        [p,2]
        \subseteq
        \operatorname{Fix}_{L}(C_\varphi)
        \subseteq
        \operatorname{Fix}_{\ell}(C_\varphi)
        \subseteq
        [p,2].
\]
Therefore
\[
        \operatorname{Fix}_{\ell}(C_\varphi)
        =
        \operatorname{Fix}_{L}(C_\varphi)
        =
        [p,2].
\]

Finally,
\[
        \CK(H^p,H^p)
        =
        \CS(H^p,H^p)
\]
follows from \cite[Corollary~1.3]{LNST2017}.
This completes the proof of Theorem~C.
\end{proof}

\section{Relative singularity classes and critical examples}
\label{sec:lens}

The class identities in Corollary~D follow from Theorems~A and~C.
It remains to prove the strict inclusions. For \(p<q\), a critical
lens map gives a bounded noncompact composition operator and hence
\(\CK(H^p,H^q)\subsetneq\CS(H^p,H^q)\).

For \(0<\alpha<1\), let
\(T(z)=(1-z)/(1+z)\), \(z\in\D\), and define the symmetric lens map
\begin{equation}
\label{eq:lens-map}
        \varphi_\alpha(z)
        =
        T^{-1}\bigl(T(z)^\alpha\bigr),
\end{equation}
where the principal branch of \(w^\alpha\) is taken on the right
half-plane. Equivalently,
\[
        \varphi_\alpha(z)
        =
        \frac{(1+z)^\alpha-(1-z)^\alpha}
        {(1+z)^\alpha+(1-z)^\alpha}.
\]

The critical lens map provides the strict inclusion when \(p<q\).

\begin{proposition}
\label{prop:lens-example}
Let \(1\le p<q<\infty\), and set \(\alpha=p/q\). Then
\(C_{\varphi_\alpha}:H^p\to H^q\) is bounded but not compact.
Consequently,
\[
        C_{\varphi_\alpha}
        \in
        \CS(H^p,H^q)\setminus\CK(H^p,H^q).
\]
\end{proposition}

\begin{proof}
Put \(\mu_\alpha=\mu_{\varphi_\alpha}\). We first determine its
boundary part. Let \(\zeta=e^{it}\in\T\setminus\{1,-1\}\), where
\(t\in(-\pi,\pi)\setminus\{0\}\). Then
\[
        T(\zeta)
        =
        -i\tan\frac{t}{2}
        \in i\mathbb R\setminus\{0\}.
\]
The branch in \eqref{eq:lens-map} extends continuously to
\(i\mathbb R\setminus\{0\}\). Hence, with \(w=T(\zeta)^\alpha\),
\[
        \varphi_\alpha^*(\zeta)=T^{-1}(w),
        \qquad
        \arg w\in
        \left\{-\frac{\alpha\pi}{2},\frac{\alpha\pi}{2}\right\}.
\]
Since \(0<\alpha<1\), we have \(\operatorname{Re}w>0\), and therefore
\[
        1-\left|\varphi_\alpha^*(\zeta)\right|^2
        =
        1-\left|\frac{1-w}{1+w}\right|^2
        =
        \frac{4\operatorname{Re}w}{|1+w|^2}
        >0.
\]
At the remaining two points,
\(\lim_{r\to1^-}\varphi_\alpha(r)=1\) and
\(\lim_{r\to1^-}\varphi_\alpha(-r)=-1\). Thus the radial boundary
function determined by these limits satisfies
\begin{equation}
\label{eq:lens-contact}
        E_{\varphi_\alpha}=\{1,-1\},
        \qquad
        (\mu_\alpha)_\T=0,
\end{equation}
where the second assertion follows from
\eqref{eq:boundary-mass}.

For \(\xi\in\T\) and \(0<h<1\), let
\[
        S(\xi,h)=\{z\in\D:|z-\xi|\le h\},
        \quad\text{and}\quad
        \rho_{\varphi_\alpha}(h)
        =
        \sup_{\xi\in\T}(\mu_\alpha)_\D(S(\xi,h)).
\]
These are the pseudo-Carleson windows and the maximal Carleson
function used in
\cite[(3.1), Lemma~3.3]{LefevreLiQueffelecRodriguezPiazza2013Lens}.
Since \((\mu_\alpha)_\T=0\) by
\eqref{eq:lens-contact}, the cited estimate gives
\[
        \rho_{\varphi_\alpha}(h)
        \asymp_\alpha
        h^{1/\alpha},
        \qquad
        0<h<1.
\]

There is an absolute constant \(c>1\) such that
\begin{equation}
\label{eq:box-window}
        S(I)\subset S(\xi_I,c|I|),
        \qquad
        S(\xi,h)\subset S(I_{\xi,ch}),
\end{equation}
whenever \(I\in\mathcal I\) and \(0<h<c^{-1}\), where \(\xi_I\) is
the midpoint of \(I\) and \(I_{\xi,ch}\) is the arc centered at
\(\xi\) with normalized length \(ch\). Hence, for \(|I|<c^{-1}\),
\[
        (\mu_\alpha)_\D(S(I))
        \le
        \rho_{\varphi_\alpha}(c|I|)
        \lesssim_\alpha
        |I|^{1/\alpha}
        =
        |I|^{q/p}.
\]
The same estimate for \(|I|\ge c^{-1}\) follows from the finiteness of
\(\mu_\alpha\). Together with \((\mu_\alpha)_\T=0\),
Proposition~\ref{prop:carleson}\textup{(i)} shows that
\(C_{\varphi_\alpha}:H^p\to H^q\) is bounded.

On the other hand, for \(0<h<c^{-1}\),
\[
\begin{aligned}
        \sup_{\substack{I\in\mathcal I\\ |I|\le ch}}
        \frac{(\mu_\alpha)_\D(S(I))}{|I|^{q/p}}
        &\ge
        \sup_{\xi\in\T}
        \frac{(\mu_\alpha)_\D(S(I_{\xi,ch}))}{(ch)^{1/\alpha}}\\
        &\ge
        c^{-1/\alpha}
        \frac{\rho_{\varphi_\alpha}(h)}{h^{1/\alpha}}
        \gtrsim_\alpha 1.
\end{aligned}
\]
Thus the vanishing condition in
Proposition~\ref{prop:carleson}\textup{(i)} fails, and
\(C_{\varphi_\alpha}:H^p\to H^q\) is not compact. Its strict
singularity follows from
Corollary~\ref{cor:p-lt-q}.
\end{proof}

\begin{proof}[Proof of Corollary~D]
Suppose first that \(q<p\). By
Theorem~\ref{thm:q-lt-p-identity}, compactness, strict singularity,
and the condition \(m(E_\varphi)=0\) are equivalent. For
\(\iota(z)=z\), the operator \(C_\iota:H^p\to H^q\) is the canonical
inclusion \(H^p\hookrightarrow H^q\). Since \(m(E_\iota)=1\),
\(C_\iota\notin\mathcal C_0(H^p,H^q)\). Therefore
\[
        \CK(H^p,H^q)
        =
        \CS(H^p,H^q)
        =
        \mathcal C_0(H^p,H^q)
        \subsetneq
        \CB(H^p,H^q).
\]

By Theorem~\ref{thm:q-lt-p-class},
a bounded composition operator fixes a copy of \(\ell^r\) if and
only if it fixes a copy of \(L^r(0,1)\). This occurs precisely when
the operator is noncompact and either
\(1<p<2\) with \(p<r\le2\), or \(p\ge2\) with \(r=2\).
For these exponents, belonging to either relative singularity class
is equivalent to compactness, so both classes coincide with
\(\CK(H^p,H^q)\). For all remaining exponents, no bounded composition
operator fixes a copy of \(\ell^r\) or \(L^r(0,1)\), and both relative
singularity classes coincide with \(\CB(H^p,H^q)\). This proves
part~\textup{(i)}.

Suppose next that \(p<q\). By
Corollary~\ref{cor:boundary-obstruction}, every bounded
composition operator \(C_\varphi:H^p\to H^q\) satisfies
\(m(E_\varphi)=0\). Corollary~\ref{cor:p-lt-q}
then gives, for every \(1\le r<\infty\),
\[
\begin{aligned}
        \CK(H^p,H^q)
        &\subsetneq
        \CS(H^p,H^q)
        =
        \CSell{r}(H^p,H^q)\\
        &=
        \CSL{r}(H^p,H^q)
        =
        \mathcal C_0(H^p,H^q)
        =
        \CB(H^p,H^q),
\end{aligned}
\]
where the strict inclusion follows from
Proposition~\ref{prop:lens-example}. This proves
part~\textup{(ii)}.

Finally, suppose that \(p=q\). By Theorem~C,
\(\CK(H^p,H^p)=\CS(H^p,H^p)\).

Consider \(\psi(z)=(1+z)/2\). Since
\[
        |\psi(e^{it})|
        =
        \left|\cos\frac{t}{2}\right|,
        \qquad
        |1-\psi(e^{it})|
        =
        \left|\sin\frac{t}{2}\right|,
\]
we have \(E_\psi=\{1\}\), and hence
\(C_\psi\in\mathcal C_0(H^p)\) and \((\mu_\psi)_\T=0\).
Let \(c\ge1\) be the absolute constant in
\eqref{eq:box-window}. For \(0<h<1/c\), let
\(I_{1,ch}\) be the arc centered at \(1\) with normalized length
\(ch\). Then \(S(1,h)\subset S(I_{1,ch})\). Moreover,
\[
\begin{aligned}
        (\mu_\psi)_\D(S(1,h))
        &=
        m\left(
        \left\{
        e^{it}:
        \left|\psi(e^{it})-1\right|\le h
        \right\}
        \right)\\
        &=
        m\left(
        \left\{
        e^{it}:
        \left|\sin\frac{t}{2}\right|\le h
        \right\}
        \right)
        =
        \frac{2}{\pi}\arcsin h.
\end{aligned}
\]
It follows that
\[
\begin{aligned}
        \frac{(\mu_\psi)_\D(S(I_{1,ch}))}{|I_{1,ch}|}
        &\ge
        \frac{(\mu_\psi)_\D(S(1,h))}{ch}\\
        &=
        \frac{2\arcsin h}{\pi ch}
        \longrightarrow
        \frac{2}{\pi c}
        \qquad
        (h\to0^+).
\end{aligned}
\]
Thus the vanishing Carleson condition in
Proposition~\ref{prop:carleson}\textup{(ii)} fails, so
\(C_\psi\) is not compact. Hence
\(\CK(H^p,H^p)\subsetneq\mathcal C_0(H^p)\).

The identity symbol \(\iota(z)=z\) satisfies \(m(E_\iota)=1\), so
\(C_\iota\in\CB(H^p,H^p)\setminus\mathcal C_0(H^p)\). Therefore
\[
        \CK(H^p,H^p)
        =
        \CS(H^p,H^p)
        \subsetneq
        \mathcal C_0(H^p)
        \subsetneq
        \CB(H^p,H^p).
\]

It remains to identify the relative fixed-copy classes. By definition,
for every bounded composition operator,
\[
\begin{aligned}
        C_\varphi\in\CSell{r}(H^p,H^p)
        &\Longleftrightarrow
        r\notin\operatorname{Fix}_{\ell}(C_\varphi),\\
        C_\varphi\in\CSL{r}(H^p,H^p)
        &\Longleftrightarrow
        r\notin\operatorname{Fix}_{L}(C_\varphi).
\end{aligned}
\]
Theorem~C shows that a composition operator fixes a copy of
\(\ell^p\) if and only if it is noncompact. Hence
\(\CSell{p}(H^p,H^p)=\CK(H^p,H^p)\).

If \(1\le p<2\) and \(p<r\le2\), or if \(p>2\) and \(r=2\),
Theorem~C shows that \(C_\varphi\) fixes a copy of \(\ell^r\) if and
only if \(m(E_\varphi)>0\). The corresponding relative class is
therefore \(\mathcal C_0(H^p)\). For all remaining exponents, no
bounded composition operator fixes a copy of \(\ell^r\), and the
relative class is \(\CB(H^p,H^p)\). This gives the asserted formula
for \(\CSell{r}(H^p,H^p)\), including \(p=2\).

The function space cases follow in the same way. If \(p=1\), a copy
of \(L^r(0,1)\) is fixed exactly when \(m(E_\varphi)>0\) and
\(1<r\le2\). If \(1<p<2\), the corresponding range is
\(p\le r\le2\). If \(p=2\), a copy of \(L^2(0,1)\) is fixed exactly
when the operator is noncompact, and no other exponent occurs. If
\(p>2\), a copy of \(L^r(0,1)\) is fixed exactly when
\(m(E_\varphi)>0\) and \(r\in\{2,p\}\). These alternatives give the
asserted formula for \(\CSL{r}(H^p,H^p)\). This proves
part~\textup{(iii)} and completes the proof of Corollary~D.
\end{proof}

\noindent {\bf Data Availability}  No data was used to support this study.\smallskip
	
\noindent {\bf Conflicts of Interest}  The authors  declare that they have no conflicts of interest.\smallskip

\noindent  {\bf Acknowledgements}  The first author was supported by the Department of Education of Guangdong
Province (Grant No.~2023KTSCX072), the Guangdong Basic and Applied Basic
Research Foundation (Grant No.~2024A1515012551), and Lingnan Normal
University (Grant No.~LT2410). The corresponding author was supported by the
National Natural Science Foundation of China (Grant No.~12371131) and the Guangdong
Basic and Applied Basic Research Foundation (Grant No.~2024A1515012404).
\end{document}